\documentclass[12pt]{scrartcl}
\usepackage{cmap}
\usepackage{sansmath}
\usepackage[T1]{fontenc}
\usepackage[utf8]{inputenc}
\usepackage{lmodern}
\usepackage{mathtools,amssymb}
\usepackage{enumitem}
\usepackage[a4paper,margin=3cm]{geometry}
\usepackage{amsthm}
\usepackage[svgnames]{xcolor}
\usepackage{booktabs}
\usepackage{siunitx}
\usepackage{graphics}
\usepackage{stmaryrd}
\usepackage[section]{placeins}
\usepackage{csquotes}
\usepackage{upgreek}

\usepackage[unicode,hyperfootnotes=false,pdftex]{hyperref}
\usepackage{hyperxmp}
\usepackage{bookmark}

\hypersetup{%
        pdftitle={Maximum likelihood drift estimation for a threshold diffusion},
	pdfauthor={Antoine Lejay; Paolo Pigato},
        pdfdate={2019-03-28},
}

\usepackage[bibencoding=latin1,backend=bibtex,uniquename=init,maxnames=10,maxcitenames=2,defernumbers=true,sortcites=true,firstinits=true,style=numeric,sorting=nyt]{biblatex}
\newtheorem{proposition}{Proposition}
\newtheorem{lemma}{Lemma}

\newtheorem{data}{Data}
\theoremstyle{definition}
\newtheorem{definition}{Definition}
\theoremstyle{remark}

\newtheorem{remark}{Remark}
\newtheorem{notation}{Notation}

\newcommand\given{\nonscript\:\delimsize\vert\nonscript\:\mathopen{}} 
\newcommand\SetSymbol[1][]{\nonscript\:#1\vert\nonscript\:\mathopen{}\allowbreak}
\DeclarePairedDelimiterX\Set[1]\{\}{\renewcommand\given{\SetSymbol[\delimsize]}#1}
\DeclarePairedDelimiterX\Prb[1](){\renewcommand\given{\SetSymbol[\delimsize]}#1}

\DeclarePairedDelimiterXPP{\ppart}[1]{}{\llparenthesis}{\rrparenthesis}{^+}{#1}
\DeclarePairedDelimiterXPP{\mpart}[1]{}{\llparenthesis}{\rrparenthesis}{^-}{#1}
\DeclarePairedDelimiterXPP{\pmpart}[1]{}{\llparenthesis}{\rrparenthesis}{^\pm}{#1}
\DeclarePairedDelimiter{\abs}{|}{|}
\DeclarePairedDelimiter{\Paren}{(}{)}
\DeclarePairedDelimiter\bra{\langle}{\rangle}

\newcommand{\convas}[1][]{\xrightarrow[#1]{\mathrm{a.s.}}}
\newcommand{\convp}[1][]{\xrightarrow[#1]{\PP}}
\newcommand{\convl}[1][]{\xrightarrow[#1]{\mathrm{law}}}

\newcommand{\LL}{\mathsf{L}}

\newcommand{\NN}{\mathbb{N}}
\newcommand{\RR}{\mathbb{R}}
\newcommand{\PP}{\mathbb{P}}
\newcommand{\EE}{\mathbb{E}}
\newcommand{\cC}{\mathcal{C}}
\newcommand{\cL}{\mathcal{L}}
\newcommand{\cCz}{\mathcal{C}_0}
\newcommand{\cG}{\mathcal{G}}
\newcommand{\cN}{\mathcal{N}}
\newcommand{\cT}{\mathcal{T}}
\newcommand{\cM}{\mathcal{M}}
\newcommand{\cR}{\mathcal{R}}
\newcommand{\cB}{\mathcal{B}}
\newcommand{\cF}{\mathcal{F}}
\newcommand{\ind}[1]{\mathbf{1}_{#1}}

\newcommand{\vd}{\,\mathrm{d}}
\newcommand{\dd}{\mathrm{d}}

\newcommand{\btrue}{b^{\mathrm{true}}}

\DeclareMathOperator{\Dom}{Dom}
\DeclareMathOperator{\Hess}{Hess}
\DeclareMathOperator{\sgn}{sgn}

\def\eqlaw{\stackrel{\mathrm{law}}{=}}
\newcommand{\eqdef}{\mathbin{:=}}

\newcommand{\disbeta}{\upbeta}

{\begin{itemize}[leftmargin=1cm,noitemsep,label={$\star$}]}{\end{itemize}}

\begin{document}

\title{
Maximum likelihood drift estimation for a threshold diffusion
}

\author{Antoine Lejay\footnote{Université de Lorraine, CNRS, Inria, IECL, F-54000 Nancy, France; \texttt{Antoine.Lejay@univ-lorraine.fr}}
\ and
Paolo Pigato\footnote{Weierstrass Institute for Applied Analysis and Stochastics, Mohrenstrasse 39, Berlin, 10117, Germany; 
E-mail: \texttt{paolo.pigato@wias-berlin.de}}}

\date{\today}

\maketitle

\begin{abstract}
We study the maximum likelihood estimator of the drift parameters of a
stochastic differential equation, with both drift and diffusion coefficients
constant on the positive and negative axis, yet discontinuous at zero. 
This threshold diffusion is called drifted Oscillating Brownian motion.  

 For this continuously observed diffusion, the maximum likelihood estimator coincide with a quasi-likelihood estimator with constant diffusion term.
We show that this estimator is the limit, as observations become dense in time, of the (quasi)-maximum likelihood estimator based on discrete observations.

In long time, the asymptotic behaviors of
the positive and negative occupation times rule the ones of the estimators.

Differently from
most known results in the literature, we do not restrict ourselves to the
ergodic framework: indeed, depending on the signs of the drift, the process may
be ergodic, transient or null recurrent. For each regime, we establish whether
or not the estimators are consistent; if they are, we prove the convergence in
long time of the properly rescaled difference of the estimators towards a normal or
mixed normal distribution. These theoretical results are backed by numerical
simulations.
\end{abstract}

\bigskip


\bigskip

\noindent{\textbf{AMS 2010: }} primary: 62M05; secondary: 62F12; 60J60.

\bigskip

\noindent{\textbf{Acknowledgement: }}
P. Pigato gratefully acknowledges financial support from ERC via Grant CoG-683164. The authors are grateful to an anonymous reviewer for their useful suggestions.

\newpage


\section{Introduction}

We consider the process, called a \emph{drifted Oscillating Brownian motion} (DOBM),
which is the solution to the Stochastic Differential equation (SDE) 
\begin{equation}
    \label{eq:process}
    \xi_t=\xi_0+\int_0^t \sigma(\xi_s)\vd W_s+\int_0^t b(\xi_s)\vd s,
\end{equation}
with
\begin{equation}
    \label{parametersDOBM}
    \sigma(x)=\begin{cases}
	\sigma_+>0&\text{ if }x\geq 0,\\
	\sigma_->0&\text{ if }x<0
    \end{cases}
    \quad\text{and}\quad
    b(x)=\begin{cases}
	b_+\in \RR &\text{ if }x\geq 0,\\
	b_-\in \RR &\text{ if }x<0.
    \end{cases}
\end{equation}
The strong existence to \eqref{eq:process} follows for example
from the results of \cite{legall}. 
Separately on $\RR_+$ and $\RR_-$, the dynamics of such process is the one of a
Brownian motion with drift, with threshold and regime-switch at $0$, consequence of the discontinuity of the coefficients.

This model can be seen as an alternative to the model studied in
\cite{motaesquivel}, which is a continuous time version of the Self-Exciting
Threshold Autoregressive models (SETAR), a subclass of the TAR models
\cite{tong1983,tong2011}.

The practical applications of such processes are numerous. In finance, 
we show in \cite{lp1} that an exponential form of this process generalizes
the Black \& Scholes model in a way to 
model leverage effects. 
Moreover, the introduction of a piecewise constant
drift such as the one in \eqref{parametersDOBM} is a straightforward way to
produce a mean-reverting process, if $b_+<0$ and $b_->0$. In \cite{lp1}, we
find some empirical evidence on financial data that this may be the case. This
corroborates other studies with different models~\cite{Monoyios:2002dk,su2015,su2017,Poterba}.  
Option pricing for this model may be performed efficiently using analytic techniques
as in \cite{gairat,lipton:2018,liptonsepp}.

Still in finance, the solution to \eqref{eq:process} models other
quantities than stocks. In \cite{Gerber}, Eq.~\eqref{eq:process} with
constant volatility serves as a model for the surplus of a company after the
payment of dividends, which are payed only if the profits of the company are
higher than a certain threshold.  Similar threshold dividend pay-out strategies are
considered in \cite{at}. In these works, the behavior of the process at the
discontinuity is referred to as \textquote{refraction}. SETAR models 
have also applications to deal with transaction costs or regulator interventions \cite{Yadav:1994kma} and to interest and exchange rates~\cite{Clements:1999cx,interestrate}.

More general discontinuous drifts and volatilities arise in
presence of Atlas models and other ranks based models
\cite{ichiba2013}. 
SDE with discontinuous coefficients have also numerous applications
in physics \cite{SATTIN20083941,ramirez_2013}, meteorology \cite{hs}
and many other domains. 

In \cite{su2015,su2017}, F.~Su and K.-S.~Chan study 
the asymptotic behavior of the quasi-likelihood estimator of a diffusion
with piecewise regular diffusivity and piecewise affine drift
with an unknown threshold. The quasi-likelihood they use is 
based on the Girsanov density where the diffusivity is replaced by~$1$.
In particular, they provide a hypothesis test to decide
whether or not the drift is affine or piecewise affine in the ergodic situation. 

In \cite{kutoyants2012}, Y. Kutoyants considers the estimation
of a threshold $r$ of a diffusion with a known or unknown drift switching at $r$.
His results are then specialized to Ornstein-Uhlenbeck type processes. 
Also this framework assumes that the diffusion is ergodic.

In the present paper, we first derive the maximum likelihood estimators for the
drift parameters $b_-$ and $b_+$ from continuous-time observations.  We then
write a quasi-maximum likelihood estimator, as in \cite{su2015,su2017}, and
show that, for a process $\xi$ as in \eqref{eq:process}, these two estimators
are actually equal.  Then, we derive the corresponding estimators based on
discrete observations, and show that as the observations become dense in time,
the estimator converges to the one based on continuous observations.

We study the asymptotic behavior of the maximum likelihood estimators as the time tends to infinity,
in order to derive some confidence intervals when available.  
This article completes \cite{LP}, where we estimate $(\sigma_-,\sigma_+)$
from high-frequency data. We use our estimators on 
financial historical data in \cite{lp1}.

Based on the maximum and quasi-maximum likelihood, our estimators of $b_\pm$ can be expressed as
\begin{equation*}
    \beta^\pm_T
    =
\pm \frac{(\pm\xi_T)\vee 0-(\pm\xi_0)\vee 0 -L_T(\xi)/2 }{Q^\pm_T },
\end{equation*}
where $Q^+_T$ (resp. $Q^-_T$) is the occupation time of the positive (resp.
negative) side of the real axis, and $L_T(\xi)$ is the symmetric local time of
$\xi$ at $0$.  The occupation times $Q^\pm_T$ can be estimated from discrete
observations of a trajectory of $\xi$ using Riemann sums.  We also propose a
new estimator of $L_T(\xi)$ from discrete observations, which can be
implemented without any previous knowledge of $\sigma_\pm$. Substituting in the
formula above $Q^\pm$ and $L_T(\xi)$ with these discrete counterparts, we
obtain the (quasi)-maximum likelihood estimators from discrete observations.
Since all these quantities converge to their continuous-time versions as the
observations become dense in time, the estimators for the drift coefficients
converge as well.

The long time asymptotic regime of the process depends on the
respective signs of the coefficients $(b_-,b_+)$. Using symmetries, 
this leads 5 different cases in which the process may be ergodic, 
null recurrent or transient and the estimators have different
asymptotic behaviors.
In some situations, 
the estimators are not convergent. In others, we establish 
consistency as well as Central Limit Theorems, 
with speed $T^{1/2}$ or $T^{1/4}$, depending again on the signs of $b_\pm$.
We summarize in Table~\ref{table:summary} the various asymptotic behaviors.
We are in a situation close to the one encountered by M.~Ben Alaya and
A.~Kebaier in \cite{benAlaya2012,alaya_kebaier_2013} for estimating square-root diffusions, where
several situations shall be treated. The works
\cite{su2015,su2017,kutoyants2012} mentioned above only consider ergodic
situations. Non-parametric estimation of the drift in the recurrent case is considered in \cite{bp}.

	\begin{table}
	    \centering
    \begin{tabular}{r r r c c  c }
	\toprule
	&&& $\beta_T^+ - b_+$  & $\beta_T^- - b_-$  \\
	\cmidrule[1pt]{4-5}
	\textbf{(E)} &  $b_+<0,b_->0$ & ergodic   & $\approx\frac{1}{\sqrt{T}}\sqrt{1-\frac{b_+}{b_-}} \sigma_+ \cN$  & $\approx\frac{1}{\sqrt{T}}\sqrt{1-\frac{b_-}{b_+}} \sigma_- \cN$ \\
	\midrule
	\textbf{(N0)}  & $b_+=0,b_-=0$& null recurrent  & $\frac{1}{\sqrt{T}}\beta_1^+$ & $\frac{1}{\sqrt{T}}\beta_1^-$ \\
	\midrule
	\textbf{(N1)} & $b_+=0,b_->0$& null recurrent  & $\approx\frac{1}{\sqrt{T}}\sigma_+ \cN^+$ & $\approx\frac{1}{T^{1/4}}\sigma_-\frac{\sqrt{b_-}}{\sqrt{\sigma_+}}\frac{\cN^-}{\sqrt{\abs{\cN}}}$ \\
	\midrule
	\textbf{(T0)}  & $b_+>0,b_-\geq 0$& transient  & $\approx\frac{1}{\sqrt{T}}\sigma_+ \cN$ & $\cR_{\mathbf{T0}}$ as $T\to\infty$\\ 
 \midrule
 \textbf{(T1)} & $b_+>0,b_-< 0$& transient   & $\approx\frac{1}{\sqrt{T}}\sigma_+ \cN$ & $\cR_{\mathbf{T1}}^+$ as $T\to\infty$ \\
	\bottomrule
    \end{tabular}
    \caption{
	\label{table:summary} 
		Asymptotic behavior of estimators, 
	where $\cN$, $\cN^+$ and $\cN^-$ are independent, unit Gaussian variables.
	The law of $(\beta^-_1,\beta^+_1)$ in case \textbf{(N0)}  is given in \eqref{explicitdensitynodrift}.
	The r.v.s $\cR_{\mathbf{T0}}$ and $\cR^+_{\mathbf{T1}}$ follow the law in \eqref{eq:dens:r}.
	 Results of both sides in \textbf{(T1)} are wrt to $\PP_+$ (cf. Proposition~\ref{prop:T1}), which intuitively can be thought as conditioning to the process diverging towards positive infinity.
}
\end{table}

Finally, we develop in Section~\ref{sec:wilk} a hypothesis test for the value of the drift
which is based on the Wilk's theorem, which relates asymptotically 
the log-likelihood to a~$\chi^2$ distribution with 2 degrees 
of freedom. Besides, we show in Section~\ref{sec:LAN} the Local Asymptotic Normality 
(LAN \cite{lecam53,lecam00}) and the Local Asymptotic Mixed Normality 
(LAMN \cite{jeganathan82a}) in the ergodic
case  and the null recurrent case with non vanishing drift. These LAN/LAMN
properties are related to the efficiency of the operators. 
The Wilk's as well as the LAN/LAMN properties are proved by combining
the quadratic nature of the log-likelihood with our martingale central limit
theorems.

\bigskip

\noindent\textbf{Outline.}  In Section~\ref{sec:mle}, we present 
the maximum and quasi-maximum likelihood estimators, which are based on the Girsanov 
transform, and their discrete-time versions. In Section~\ref{sec:regime}, we characterize the different 
regimes of the process accordingly to the signs of the drifts. 
Our main results are presented in Section~\ref{sec:CLT}.
The limit theorems that we use are presented in Section~\ref{sec:aux}.
The proofs for each cases are detailed in Section~\ref{sec:proofs}. 
We present the Wilk theorem and the LAN/LAMN property in Section~\ref{sec:loglik}. 
Finally, in Section \ref{simulations}, we conclude this article with numerical experiments.



\section{Maximum and quasi-maximum likelihood estimators} \label{sec:mle}

In this section, we construct two estimators for the parameters $(b_-,b_+)$ of the drift coefficient of $\xi$. The maximum likelihood estimator based on continuous time observations can be constructed if we assume to observe a continuous time path $(\xi_t)_{t\in[0,T]}$. A quasi-maximum likelihood estimator can also be constructed, and from the fact that $b(\cdot)$ and $\sigma(\cdot)$ are constant above and below the threshold, it follows that
these estimators are in fact the same one. Then, we construct the analogous estimator based on discrete time observations of the process. In the end we will see that, as observations become dense in time, the discrete time estimator converges to the one in continuous time.

\subsection{The maximum and quasi-maximum likelihood estimator based on continuous time observations}

First, we propose and discuss the maximum likelihood estimator from continuous time observations. 

\begin{data}
    \label{dat:1}
We observe a path $(\xi_t)_{t\in[0,T]}$
on the time interval $[0,T]$ of one solution to~\eqref{eq:process}, together
with its negative and positive occupation times 
\begin{equation}
    \label{eq:def:Q}
    Q^-_T=\int_0^T \ind{\xi_s<0}\vd s\text{ and } 
    Q^+_T=\int_0^T \ind{\xi_s>0}\vd s, 
\end{equation}
as well as its symmetric local time 
\begin{equation*}
    L_T(\xi)=\lim_{\epsilon\to 0}\frac{1}{2\epsilon}\int_0^T
    \ind{-\epsilon\leq \xi_s\leq \epsilon}\vd \bra{\xi}_s.
\end{equation*}
\end{data}

The question of the discretization of $Q^\pm_T$ as well as $L_T(\xi)$ 
is detailed in Sections~\ref{sec:disc} and ~\ref{sec:discrete}.

We define
\begin{align}
    \label{eq:def:M}
    M^\pm&\eqdef \int_0^\cdot \sigma_\pm\mathbf{1}_{\pm \xi_s\geq 0}\vd W_s
    \\
    \label{eq:def:R}
\text{and } 
    R^\pm_T&\eqdef \int_0^T \ind{\pm\xi_s\geq 0}\vd  \xi_s.
\end{align}
The quantities $M^\pm$ are are continuous time martingales    
with $\bra{M^\pm}=\sigma_\pm^2 Q^\pm$ and $\bra{M^+,M^-}=0$. Moreover, 
\begin{equation}
    \label{eq:RM}
R^\pm_T=M^\pm_T+b_\pm Q^\pm_T.
\end{equation}

Noticing that the occupation time is non decreasing, for the sake of simplicity we write
\begin{equation*}
    Q^\pm_\infty\eqdef\lim_{T\to\infty} Q_T^\pm\in \RR_+ \cup \{\infty\}.
\end{equation*}

The Girsanov weight of the distribution of \eqref{eq:process}
with respect the driftless ($b_+=b_-=0$) SDE is 
\begin{equation}
    \label{eq:girsanov}
    G_T(b_+,b_-)=\exp\left(\int_0^T\frac{b(\xi_s)}{\sigma^2(\xi_s)}\vd\xi_s
    -\frac{1}{2}\int_0^T\frac{b^2(\xi_s)}{\sigma^2(\xi_s)}\vd s\right).
\end{equation}
With the expression of $Q^\pm$ and $R^\pm$ given by \eqref{eq:def:Q} and \eqref{eq:def:R}, 
this Girsanov weight can be expressed simply as 
\begin{equation*}
    G_T(b_+,b_-)=\exp\left(\frac{b_+}{\sigma_+^2}R^+_T
	+\frac{b_-}{\sigma_-^2}R^-_T
	-\frac{b_+^2}{\sigma_+^2}Q_T^+
    -\frac{b_-^2}{\sigma_-^2}Q_T^-\right).
\end{equation*}

A reasonable way to set up an estimator of $(b_-,b_+)$ is to consider
$G_T(b_+,b_-)$ as a likelihood and to optimize this quantity. This
is how estimators for the drift are classically constructed~\cite{kut,lipster_II}.
We also set the following function, defined similarly as in \cite{su2015} as 
\begin{align*}
    \Lambda_T(b_+,b_-)&=\int_0^T b(\xi_s)\vd \xi_s-\frac{1}{2}\int_0^T b(\xi_s)^2\vd s
    \\
    &=b_+ R^+_T+b_- R^-_T-\frac{1}{2}b_+^2 Q_T^+-\frac{1}{2}b_-^2 Q_T^-.
\end{align*}
Such function can be interpreted
as a quasi-likelihood function, and has the advantage wrt \eqref{eq:girsanov} of not involving the diffusion coefficient. This approach can be used also when a specific functional form for the diffusion coefficient $\sigma$ is not specified.

\begin{notation}
To avoid confusion with the $+$ and $-$ used as indices, 
we write $\ppart{x}\eqdef\max\Set{x,0}$ for the positive part
of $x$ and $\mpart{x}\eqdef\max\Set{-x,0}\geq 0$ for the negative part.
\end{notation}

\begin{proposition} 
    \label{prop:qmle}
   Both the 
likelihood $G_T(b_+,b_-)$  and the    
   quasi-likelihood $\Lambda_T(b_+,b_-)$ are maximal at $(\beta^+_T,\beta^-_T)$ given by 
\begin{equation*}
    \beta^+_T=\frac{R_T^+}{Q_T^+}=b_++\frac{M_T^+}{Q_T^+}
    \text{ and }
    \beta^-_T=\frac{R_T^-}{Q_T^-}=b_-+\frac{M_T^-}{Q_T^-}.
\end{equation*}   
Estimators
$(\beta^+_T,\beta^-_T)$ can also be expressed as
    \begin{equation}
	\label{estb}
	\beta^\pm_T
    = \pm \frac{\pmpart{\xi_T}-\pmpart{\xi_0} -L_T(\xi)/2 }{Q^\pm_T }.
    \end{equation}
\end{proposition}

\begin{proof}
The equation $\nabla \Lambda_T(b_+,b_-)=0$ is solved for $(\beta^+_T,\beta^-_T)$ as above.
Since $Q_T^\pm\geq 0$, $\Lambda_T(b_+,b_-)$ is not only critical but also maximal at $(\beta^+_T,\beta^-_T)$
with $\beta^\pm_T=R_T^\pm/Q_T^\pm$. With \eqref{eq:RM}, we obtain 
$\beta^\pm_T=b_\pm+M_T^\pm/Q_T^\pm$.
The same holds for  $\log G_T(b_+,b_-)$.

By the Itô-Tanaka formula, 
\begin{equation*}
    \ppart{\xi}_T=\ppart{\xi}_0+\int_0^t \ind{\xi_s\geq 0}\vd \xi_s+\frac{1}{2}L_T(\xi)
    \text{ and }
    \mpart{\xi}_T=\mpart{\xi}_0-\int_0^t \ind{\xi_s\leq 0}\vd \xi_s+\frac{1}{2}L_T(\xi). 
\end{equation*}
Hence \eqref{estb}. 
\end{proof}


\subsection{An estimator based on discrete observations}

\label{sec:disc}

Within the framework of Data~\ref{dat:1}, we constructed estimator $\beta$ in Proposition \eqref{prop:qmle}. In practice, however, one cannot observe the whole trajectory $(\xi_t)_{t\in [0,T]}$. We explain in Section~\ref{simulations} how the quantities involved in the estimator expressed as in \eqref{estb} can be approximated from discrete observations of the process. Hence, an approximation of the estimators $\beta^\pm_T$ may be 
constructed from the discrete observations $(\xi_{iT/N})_{i=0,\dotsc,N}$. 

We now consider an alternative framework, in which we suppose from the beginning that the path is observed only on a discrete time grid. 

\begin{data}
    \label{dat:2}
    We observe of a path $(\xi_{iT/N})_{i=0,\dotsc,N}$ at discrete times.
\end{data}

Owing to the expressions for $G_T(b_+,b_-)$ and $\Lambda_T(b_+,b_-)$, we discretize the stochastic
integrals as well as the occupation times. 
With  $\xi_i\eqdef \xi_{iT/N}$ and $\Delta_i \xi=\xi_{i+1}-\xi_i$, we set 
\begin{equation}
    \label{eq:dis:QR}
    \mathsf{Q}^+_{T,N} \eqdef \frac{T}{n}\sum_{i=0}^{N-1}  \ind{\xi_{i}\geq 0} 
       \text{ and }
   \mathsf{R}^\pm_{T,N} \eqdef \sum_{i=0}^{N-1} \ind{\pm\xi_i\geq 0}\Delta_i \xi
\end{equation}
which are discretized versions of $Q^\pm_T$ and $R^\pm_T$. We also define 
the \emph{discretized likelihood} as 
\begin{equation*}
    \mathsf{G}_{T,N}(b_+,b_-)=
   \exp\left(
   \sum_{i=0}^{N-1} \frac{b(\xi_i)}{\sigma^2(\xi_i)} \Delta_i \xi-\frac{1}{2N}\sum_{i=0}^{N-1} \frac{b^2(\xi_i)}{\sigma^2(\xi_i)}
\right).
\end{equation*}
and the \emph{discretized quasi-likelihood} as 
\begin{equation*}
    \mathsf{\Lambda}_{T,N}(b_+,b_-)=\sum_{i=0}^{N-1} b(\xi_i)\Delta_i \xi-\frac{1}{2N}\sum_{i=0}^{N-1} b(\xi_i)^2.
\end{equation*}
We also  define the following discrete-time approximation of the local time: 
\begin{equation}
    \label{eq:dis:loctime}
\LL_{T,N}\eqdef 2 \sum_{i=0}^{N-1} \ind{\xi_i \xi_{i+1}<0} \abs{\xi_{i+1}}.
\end{equation}

We have a result similar to the one of Proposition~\ref{prop:qmle}. 

\begin{proposition} 
    \label{prop:qmle:2}
    The likelihood $\mathsf{G}_{T,N}(b_+,b_-)$ and the quasi-likelihood $\mathsf{\Lambda}_{T,N}(b_+,b_-)$ are maximal at 
    $(b_+,b_-)=(\disbeta^+_{T,N},\disbeta^-_{T,N})$, with
\begin{equation}
    \label{def:b:qml}
    \disbeta_{T,N}^+\eqdef \frac{\mathsf{R}^+_{T,N}}{\mathsf{Q}^+_{T,N}}
    \text{ and }
    \disbeta_{T,N}^-\eqdef \frac{\mathsf{R}^-_{T,N}}{\mathsf{Q}^-_{T,N}}.
\end{equation}
We also have the following discrete version of \eqref{estb}:
\begin{equation}
    \label{eq:disbeta}
    \disbeta^\pm_{T,N}=\pm\frac{\pmpart{\xi}_T-\pmpart{\xi}_0-\LL_{T,N}/2}{\mathsf{Q}^\pm_{T,N}}.
\end{equation}
\end{proposition}
\begin{proof}

    The fact that $\mathsf{G}_{T,N}(b_+,b_-)$ and $\mathsf{\Lambda}_{T,N}(b_+,b_-)$ are maximal at $\disbeta_{T,N}^\pm$ is trivial.
We note that 
\begin{equation*}
    \ind{\xi_i\geq 0}\Delta_i \xi=\ppart{\xi_{i+1}}-\ppart{\xi_i}-\ind{\xi_i\xi_{i+1}<0}\abs{\xi_{i+1}}
\end{equation*}
so that 
\begin{equation}
    \label{eq:new:7}
    \mathsf{R}^+_{T,N}=\ppart{\xi}_T-\ppart{\xi}_0-\frac{1}{2}\LL_{T,N}.
\end{equation}
Similarly, one can show that
\begin{equation}
    \label{eq:new:7bis}
    \mathsf{R}^-_{T,N}=-\mpart{\xi}_T+\mpart{\xi}_0+\frac{1}{2}\LL_{T,N}.
\end{equation}
This proves \eqref{eq:disbeta}.
\end{proof}

\begin{lemma}
    \label{lem:discrete}
    For any $T>0$, $\mathsf{Q}^\pm_{T,N}$ converges in probability to $Q^\pm_T$ and
    \begin{equation}
	\label{eq:new:6}
	\mathsf{R}^\pm_{T,N}\xrightarrow[N\to\infty]{\PP} R_T^\pm=\int_0^T \ind{\pm\xi_s\geq 0}\vd \xi_s
	=\pm\Paren*{\pmpart{\xi_T}-\pmpart{\xi_0}-\frac{1}{2}L_T(\xi)}.
    \end{equation}
    In particular, $\disbeta_{T,N}^\pm$ converges in probability to $\beta^\pm_T$ as $N\to\infty$.
\end{lemma}

\begin{remark} 
The result of Lemma \ref{lem:discrete} says that, for fixed $T$,
the approximation of the continuous-time estimator $\beta^\pm_T$ given by
$\disbeta_{T,N}^\pm$ becomes more accurate as $N$ increases. It would be
interesting, from a practical viewpoint, to study the asymptotic behavior of
these estimators as $T,N$ jointly increase in a suitable way.
This seems to be a very hard problem, even in the simplified setting with
a constant volatility coefficient and drift as in the ergodic case \textbf{(E)}.
Indeed, the speed of the convergence in \eqref{eq:new:6} can be
obtained for fixed $T$, but some sort of uniformity in $T$ of the convergences
in Lemma~\ref{lem:discrete} is needed to prove a joint result in
$T,N$. Unluckily, even in this simplified case, the asymptotic behavior of the local time
estimator depends on the very technical results of Jacod (see \cite[Theorem
1.2]{j1}), which are not uniform in $T$. Besides, they are proved using a Girsanov transform in order to
eliminate the drift. As $T$ increases, the Girsanov weight degenerates and it
is very difficult to control its behavior. In the general framework with
$\sigma_+\neq \sigma_-$, this problem looks even harder since the result in
\cite{j1} does not directly apply. 
\end{remark}

\begin{proof}
As we are considering a convergence in probability, due to the Girsanov theorem, we may 
assume that $b_+=b_-=0$. For the sake of simplicity, we assume that $\xi_0=0$.
The convergence of $\mathsf{Q}^\pm_{T,N}$ to $Q^\pm_T$ follows from \cite[Theorem~4.14, p.~3587]{lmt1}.

We consider now only the convergence of $\mathsf{R}^+_{T,N}$ to $R_T^+$, the computations being
the same for the \textquote{negative} part.

Besides, we know that $\xi_t=\sigma(X_t)X_t$, where $X$ 
is a Skew Brownian motion, \textit{i.e.}, the solution to the SDE
$X_t=W_t+\theta L_t(X)$, where $W$ is a Brownian motion, $L_t(X)$ is the local 
time of $X$ at $0$ and $\theta=(\sigma_--\sigma_+)/(\sigma_-+\sigma_+)$. The local
times $L(\xi)$ and $L(X)$ are linked by (see \cite{LP}) 
\begin{equation}
    \label{eq:new:2}
    L(X)=\frac{\sigma_++\sigma_-}{2\sigma_+\sigma_-}L(\xi).
\end{equation}
The density of the skew Brownian motion is 
\begin{equation*}
    p_\theta(t,x,y)
    =\frac{1}{\sqrt{2\pi t}}\exp\Paren*{-\frac{(x-y)^2}{2t}}
    +\theta\sgn(y)\frac{1}{\sqrt{2\pi t}}\exp\Paren*{-\frac{(\abs{x}+\abs{y})^2}{2t}}.
\end{equation*}

Using this transform, with $f(x,y)\eqdef\sigma(y)\ind{xy<0}\abs{y}$, 
\begin{equation*}
\frac{\LL_{T,N}}{2}
    =\frac{1}{\sqrt{N}} \sum_{i=0}^{N-1} \ind{N X_i X_{i+1}<0} \sqrt{N}\abs{X_{i+1}}\sigma(X_{i+1})
    =\frac{1}{\sqrt{N}} \sum_{i=0}^{N-1}f(\sqrt{N}X_i,\sqrt{N}X_{i+1}).
    \end{equation*}
    Following \cite{lmt1}, which adapts to the Skew Brownian motion some results of \cite{j1}, we define
\begin{equation*}
    F(x)\eqdef\int_{-\infty}^{+\infty} p_\theta(1,x,y)f(x,y)\vd y.
\end{equation*}
Hence, for $x<0$, 
\begin{align*}
    F(x)&=\sigma_+(1+\theta)\int_0^{+\infty} 
\frac{1}{\sqrt{2\pi}}\exp\Paren*{-\frac{(x-y)^2}{2}}y\vd y
\\
&=\sigma_+\frac{(1+\theta)}{\sqrt{2\pi}}\exp\Paren*{\frac{-x^2}{2}}
+\sigma_+(1+\theta)x\overline{\Phi}(-x)
\end{align*}
with $\overline{\Phi}(x)=\int_x^{+\infty} \exp(-z^2)/\sqrt{2\pi}\vd z$. 
Similarly, for $x>0$, 
\begin{equation*}
    F(x)=+\sigma_-\frac{(1-\theta)}{\sqrt{2\pi}}\exp\Paren*{\frac{-x^2}{2}}
    -\sigma_-(1-\theta)x\overline{\Phi}(x).
\end{equation*}
Let us define
\begin{equation*}
    C
    \eqdef(1+\theta)\int_0^{+\infty} F(x)\vd x+(1-\theta)\int_{-\infty}^0 F(x)\vd x
    =(1-\theta^2)\frac{\sigma_-+\sigma_+}{4}.
\end{equation*}
Injecting the value of $\theta$, 
\begin{equation}
    \label{eq:new:3}
    C=\frac{\sigma_-\sigma_+}{\sigma_-+\sigma_+}.
\end{equation}
From \cite[Proposition~2]{lmt1}, using \eqref{eq:new:2} and \eqref{eq:new:3},
\begin{equation*}
\frac{1}{2}\LL_{T,N} \xrightarrow[N\to\infty]{\PP} CL_T(X)=\frac{1}{2}L_T(\xi).
\end{equation*}
This,  with \eqref{eq:new:7}, shows \eqref{eq:new:6}.
\end{proof}


\section{Analytic characterization of the regime of the process}

\label{sec:regime}

\subsection{Scale function and speed measure}

A well known fact \cite{ito,ks,rogers2} states 
that the infinitesimal generator $(\cL,\Dom(\cL))$ of the process $\xi$ solution to \eqref{eq:process}
may be written as 
\begin{gather*}
    \cL f=\frac{1}{2}\sigma^2(x)e^{-h(x)}\frac{\dd}{\dd x}\left(
    e^{h(x)}\frac{\dd f(x)}{\dd x}\right)\ 
    \text{ with }
    h(x)=\int_0^x \frac{2b(y)}{\sigma^2(y)}\vd y\\
    \text{ for all }
    f\in\Dom(\cL)=\Set{f\in\cCz(\RR)\given \cL f\in\cCz(\RR)}.
\end{gather*}
The process $X$ is fully characterized by its \emph{speed measure} $M$
with a density $m$ and its \emph{scale function} $S$ with 
\begin{equation}
    \label{eq:speend-n-scale}
    m(x)\eqdef\frac{2}{\sigma(x)^2}\exp( h(x) )
\text{ and }
S(x)\eqdef\int_0^x \exp( -h(y) )\vd y.
\end{equation}

\subsection{The regimes of the process}

The diffusion $X$ is either recurrent or transient.
If $\lim_{x\to+\infty} S(x)=+\infty$ and $\lim_{x\to-\infty} S(x)=-\infty$,
then the process is (positively or null) \emph{recurrent}. 
Otherwise, it is \emph{transient} \cite{ito,ks}.

When $b(x)=b_+$ for $x\geq 0$, 
\begin{equation*}
    S(x)=\begin{cases}
	x&\text{ if }b_+=0,\\
	\dfrac{\sigma_+^2}{2b_+}\left(1-\exp\left(-\dfrac{2b_+x}{\sigma_+^2}\right)\right)&\text{ if }b_+>0,\\
	\dfrac{\sigma_+^2}{2\abs{b_+}}\left(\exp\left(\dfrac{2\abs{b_+}x}{\sigma_+^2}\right)-1\right)&\text{ if }b_+<0.
    \end{cases}
\end{equation*}
Similar formulas hold for~$b_-$. 
Hence, the process~$\xi$ is transient if only if~$b_+>0$ or~$b_-<0$. 

A recurrent process is either \emph{null recurrent} or \emph{positive recurrent}. 
The process is positive recurrent if and only if $M(\RR)\eqdef\int_\RR m(x)\vd x<+\infty$,
in which case it is actually \emph{ergodic}.
Therefore, the process $\xi$ is ergodic if and only if~$b_+<0$ and~$b_->0$. Otherwise, 
the process $\xi$ is only null recurrent. 

When the process is ergodic ($b_+<0$, $b_->0$), its invariant measure is 
\begin{equation*}
    \frac{m(x)}{M(\RR)}\vd x
    =\begin{cases}
	\dfrac{1}{\sigma_+^2}\times\dfrac{|b_+| b_-}{b_-+\abs{b_+}}e^{\frac{-2x\abs{b_+}}{\sigma_+^2}}&\text{ if }x\geq0,\\
		\dfrac{1}{\sigma_-^2}\times
		\dfrac{b_-\abs{b_+}}{b_-+\abs{b_+}}e^{\frac{2xb_-}{\sigma_-^2}}&\text{ if }x<0.
    \end{cases}
\end{equation*}

Therefore, the regime of $\xi$ depends only on the respective signs of~$b_+$ and~$b_-$. 
Nine combinations are possible. As some cases are symmetric,
we actually consider five cases exhibiting different asymptotic behaviors of $Q_T^\pm$,
hence of the estimators.
This is summarized in Table~\ref{table:rec}.

\begin{table}
    \begin{center}
    \begin{tabular}{c | c c c}
	& $b_+>0$ & $b_+=0$ &  $b_+<0$ \\
	\toprule
    $b_->0$ & transient \textbf{T0} & null recurrent \textbf{N1} & ergodic \textbf{E} \\
$b_-=0$ & transient \textbf{T0} & null recurrent \textbf{N0} & null recurrent \textbf{N1}\\
	$b_-<0$ & transient \textbf{T1} & transient \textbf{T0}& transient \textbf{T0}\\
	\bottomrule
    \end{tabular}
    \caption{\label{table:rec} Recurrence and transience properties of $\xi$.}
\end{center}
\end{table}

These cases are:
\begin{itemize}[noitemsep,topsep=0pt]
    \item[\textbf{E)}] Ergodic case $b_+<0$, $b_->0$.
    \item[\textbf{N0)}] Null recurrent case $b_+=0$, $b_-=0$.
    \item[\textbf{N1)}] Null recurrent case $b_+=0$, $b_->0$.
    \item[\textbf{T0)}]  Transient case $b_+>0$, $b_-\geq 0$.
    \item[\textbf{T1)}]  Transient case $b_+>0$, $b_-< 0$.
\end{itemize}
Case \textbf{T0} corresponds to two entries of  table \ref{table:rec}.
The case $b_+<0$, $b_-=0$ is symmetric to \textbf{N1}.
Case $b_+\leq 0$, $b_-<0$ is symmetric to \textbf{T0}.

\section{Asymptotic behavior of the estimators}
\label{sec:CLT}

In this section, we state our main results on the asymptotic behavior  of the occupation times of the process and the corresponding ones of the estimators, for each of the 5 cases.

\begin{proposition}[Ergodic case \textbf{E}]
    \label{prop:E}
 If $b_+<0,b_->0$, then 
    \begin{equation*}
\left(\frac{Q^+_T}{T},\frac{Q^-_T}{T}\right) \convas[T\to\infty]\left(\frac{\abs{b_-}}{\abs{b_-}+\abs{b_+}},\frac{\abs{b_+}}{\abs{b_-}+\abs{b_+}}\right).
    \end{equation*}
    In addition, 
    \begin{gather*}
	(\beta^+_T,\beta^-_T)\convas[T\to\infty] (b_+,b_-)\\
	\text{ and }
	\frac{\sqrt{T}}{\sqrt{\abs{b_-}+\abs{b_+}}}(\beta^+_T-b_+,\beta^-_T-b_-)\convl[T\to\infty] 
	\left(
	    \frac{\sigma_+}{\sqrt{\abs{b_-}}}\cN^+,
	\frac{\sigma_-}{\sqrt{\abs{b_+}}}\cN^-\right),
    \end{gather*}
    where $\cN^+$ and $\cN^-$ are two independent, unit Gaussian random variables.
\end{proposition}

\begin{proposition}[Null recurrent case with vanishing drift \textbf{N0}]
 Assume $b_+=b_-=0$.  Assume $\xi_0=0$.
 Then 
 \begin{equation*}
\left(\frac{Q^+_T}{T},\frac{Q^-_T}{T}\right)\eqlaw (\Lambda, 1-\Lambda)\text{ for all }T>0,
\end{equation*}
where $\Lambda$ follows a law of arcsine type with density
\begin{equation*}
p_{\Lambda}(u)\eqdef
\frac{1}{\pi} \frac{1}{\sqrt{u(1-u)}} \frac{\sigma_+/\sigma_-}{1-(1-(\sigma_+/\sigma_-)^2)u} \quad \text{ for } 0<u<1.
\end{equation*}
Besides, 
\begin{equation}\label{eq:scaling:N0}
\sqrt{T}(\beta_T^+,\beta_T^-)
\eqlaw
(\beta_1^+,\beta_1^-)
\end{equation}
where the explicit joint density of $(\beta_1^+,\beta_1^-)$ is given 
by \eqref{explicitdensitynodrift} below. In particular, 
$(\beta_T^+,\beta_T^-)$ converges almost surely to $(b_+,b_-)=(0,0)$.
\end{proposition}

\begin{proposition}[Null recurrent case with non-vanishing drift \textbf{N1}]
Assume $b_+=0$, $b_->0$. Then 
\begin{equation*}
    \frac{Q^+_T}{T}\convas[T\to\infty]{}1
    \text{ and }(\beta^+_T,\beta^-_T)\convas[T\to\infty] (b_+,b_-).
\end{equation*}
In addition, there exists three independent unit Gaussian random variables
$\cN^-$, $\cN^+$ and $\cN$ such that 
\begin{equation}\label{eq:N1}
    \left(
	\frac{Q^-_T}{\sqrt{T}},
	\sqrt{T}(\beta^+_T-b_+),
    T^{1/4}(\beta^-_T-b_-)
\right)
	\convl[T\to\infty]
	\left(
	    \frac{\sigma_+}{b_-}\abs{\cN},
    \sigma_+\cN^+,
    \sigma_-\frac{\sqrt{b_-}}{\sqrt{\sigma_+\vphantom{\abs{\cN}}}}\cdot \frac{\cN^-}{\sqrt{\abs{\cN}}}
\right).
\end{equation}
\end{proposition}

\begin{remark}
    This case exhibits two different rates of convergence. This is due 
    to the fact below the threshold that the particle is pushed up, but only behave
    like a Brownian motion when above. Below the threshold, we are in a similar
    situation as for the positive recurrent case, while above it is like the null
    recurrent case.
\end{remark}

\begin{proposition}[Transient case for upward drift \textbf{T0}]
Assume $b_+>0$, $b_-\geq 0$ so that the process $\xi$ is transient
and $\lim_{T\to\infty} \xi_T=+\infty$. 
Then $Q^+_T/T$ converges almost surely to $1$ as $T\to\infty$ and 
\begin{equation}\label{eq:T0}
    \beta^+_T\convas[T\to\infty] b_+
    \text{ and }
    \sqrt{T}(\beta^+_T-b_+)\convl[T\to\infty] \sigma_+\cN^+
\end{equation}
for a unit Gaussian random variable $\cN^+$. 
Let $\ell_0$ be the last passage time to $0$, which is almost surely finite. 
Assume $\xi_0=0$. We have
\begin{equation}
    \label{eq:finite:T0}
    \beta^-_T\ind{T>\ell_0}
    =\cR_{\mathbf{T0}}\ind{T>\ell_0}
    \text{ and }
    \lim_{T\rightarrow \infty}    \beta^-_T=\cR_{\mathbf{T0}}
    \text{ a.s. with }
    \cR_{\mathbf{T0}}\eqdef\frac{L_\infty(\xi)}{2Q^-_{\ell_0}}=\frac{L_\infty(\xi)}{2Q^-_\infty}.
\end{equation}
The density of $\cR_{\mathbf{T0}}$ is given by \eqref{eq:dens:r} below.
The case $b_+\leq 0$, $b_-<0$ is treated by symmetry.
\end{proposition}

\begin{proposition}[Transient case for diverging drift \textbf{T1}]
    \label{prop:T1}
Assume $b_+>0$, $b_-<0$ so that the process $\xi$ is transient. Assume that $\xi_0=0$.
Then there exists a Bernoulli random variable $\cB\in\Set{0,1}$  such that  
\begin{gather*}
    \PP\Prb{\cB=1}=1-\PP\Prb{\cB=0}=\frac{\sigma_-b_+}{\sigma_+b_-+\sigma_-b_+}, \\
    \PP_+\Prb*{\frac{Q^+_T}{T}\xrightarrow[T\rightarrow \infty]{} 1}=1
    \text{ and }
    \PP_-\Prb*{\frac{Q^-_T}{T}\xrightarrow[T\rightarrow \infty]{} 1}=1
    \\
    \text{with }
    \PP_+\Prb{\cdot}=\PP\Prb{\cdot\given\cB=1}\text{ and }\PP_-\Prb{\cdot}=\PP\Prb{\cdot\given\cB=0}.
\end{gather*}
On the event $\Set{\cB=1}$ (resp. $\Set{\cB=0}$), $\beta^+_T$ (resp. $\beta^-_T$)
converges almost surely to~$b_+$ (resp.~$b_-$) while $\beta^-_T-b_-$
(resp. $\beta^+_T-b_+$) is the ratio of two a.s. finite random variables.

In addition, for unit Gaussian random variables $\cN^+$ and $\cN^-$, 
\begin{align}\label{eq:T1p}
    \sqrt{T}(\beta_T^+-b_+)&
\convl[T\to\infty]{} \sigma_+ \cN^+
    \text{ under }\PP_+,
    \\
    \sqrt{T}(\beta_T^--b_-)&
\convl[T\to\infty]{} \sigma_- \cN^-
    \text{ under }\PP_-. 
    \label{eq:T1n}
\end{align}
In addition,
\begin{align}\label{eq:finite:T1p}
    \lim_{T\rightarrow \infty}    \beta^-_T=\cR^-_{\mathbf{T1}}
    \text{ a.s. under }\PP_+
    \text{ with }\cR_{\mathbf{T1}}^-=\frac{L_\infty(\xi)}{2Q^-_{\infty}},
    \\
    \lim_{T\rightarrow \infty}    \beta^+_T=-\cR^+_{\mathbf{T1}}
    \text{ a.s.  under }\PP_- 
    \text{ with }\cR_{\mathbf{T1}}^+=\frac{L_\infty(\xi)}{2Q^+_{\infty}}.
    \label{eq:finite:T1n}
\end{align}
The distribution of $\cR^-_{\mathbf{T1}}$ is that of 
of $\cR$ given \eqref{eq:dens:r} below. That of $\cR^+_{\mathbf{T1}}$
is found by symmetry.
\end{proposition}


\section{Auxiliary tools}

\label{sec:aux}

In this section, we give first some results on a martingale central limit theorem
that will be used constantly. 
To deal with the transient or null recurrent cases, we make use 
of some analytic properties of one-dimensional diffusions.

\subsection{Limit theorems on martingales}

The following result follows immediately from \cite[Proposition~1, p.~148; Theorem~1, p.~150]{lepingle78a}. 

\begin{proposition}[A criterion for convergence] 
    \label{prop:convergence}
Under the true probability $\PP$, 
\begin{enumerate}[leftmargin=1em,label=\textnormal{(\roman*)},topsep=0pt,noitemsep]
    \item as $T\to\infty$, $\beta_T^+$ (resp. $\beta^-_T$) converges a.s. to
	$b_+$ (resp. $b_-$) on the event $\Set{Q^+_\infty=+\infty}$ (resp.
	$\Set{Q^-_\infty=+\infty}$).
    \item as $T\to\infty$, $M^+_T$ (resp. $M^-_T$) converges a.s. to a finite value on 
	the event $\Set{Q^+_\infty<+\infty}$ (resp.  $\Set{Q^-_\infty<+\infty}$).
	In other words, $\beta^\pm_T$ is not a consistent estimator on $\Set{Q^\pm_\infty<+\infty}$.
    \end{enumerate}
\end{proposition}

We now state an instance of a Central Limit theorem for martingales
which follows from \cite{crimaldi}. This theorem will be used to deal
with the cases \textbf{E}, \textbf{N1} and~\textbf{T1}.
Let us start by recalling the notion of stable convergence
introduced by A.~Rényi~\cite{renyi,js}.

\begin{definition}[{Stable convergence}]
    A sequence $(X_n)_{n\in\NN}$ on a probability space $(\Omega,\cF,\PP)$
    is said to \emph{converge stably} with respect to a $\sigma$-algebra $\cG\subset\cF$
    if for any bounded, continuous function $f$ and any bounded, $\cG$-measurable
    random variable~$Y$,
    \begin{equation*}
	\EE\Prb{f(X_n)Y}\xrightarrow[n\to\infty]{} \EE\Prb{f(X)Y}.
    \end{equation*}
\end{definition}

\begin{proposition}[A central limit theorem for martingales]
    \label{prop:clt}
Let $(\Omega,\cF,\PP)$ be the underlying probability space of the process $\xi$
with a filtration $(\cF_t)_{t\geq 0}$.
If for some constants $c_+,c_->0$, 
\begin{equation*}
    \frac{Q^+_T}{T}\convp[T\to+\infty] c_+\text{ and }\frac{Q^-_T}{T}\convp[T\to+\infty]c_-, 
\end{equation*}
then for the martingales $M^\pm$ defined by \eqref{eq:def:M},
\begin{equation}
    \label{eq:3}
    \left(\frac{M_T^+}{\sqrt{T}},\frac{M_T^-}{\sqrt{T}}\right)
    \xrightarrow[T\to\infty]{\cF_\infty\mathrm{-stably}} 
    (\sigma_+\sqrt{c_+}\cN^+,\sigma_-\sqrt{c_-}\cN^-),
\end{equation}
on a probability space $(\Omega',\cF',\PP')$ extending $(\Omega,\cF,\PP)$
and containing  two independent unit Gaussian random variables $\cN^+$, $\cN^-$,
themselves independent from $\xi$.
In addition, 
\begin{equation}
    \label{eq:4}
    \sqrt{T}\left(\frac{M_T^+}{Q^+_T},\frac{M_T^-}{Q^-_T}\right)
    \xrightarrow[T\to\infty]{\cF_\infty\mathrm{-stably}} 
    \left(\frac{\sigma_+}{\sqrt{c_+}}\cN^+,\frac{\sigma_-}{\sqrt{c_-}}\cN^-\right),
\end{equation}
\end{proposition}
\begin{proof}
    Set 
    \begin{equation*}
    a_T\eqdef\begin{bmatrix} 1/\sqrt{T} & 0 \\ 0 & 1/\sqrt{T}\end{bmatrix}
	\text{ and }
	q_T=\bra{M,M}_T=\begin{bmatrix} \sigma_+^2 Q_T^+ & 0 \\ 0 & \sigma_-^2 Q_T^-\end{bmatrix}.
    \end{equation*}
    Thus, 
    \begin{equation}
    \label{eq:5bis}
a_T q_T a_T'=
\begin{bmatrix}
    \sigma_+^2\frac{Q^+_T}{T} & 0 \\
    0 & \sigma_-^2\frac{Q^-_T}{T}
\end{bmatrix}
\convp[T\to\infty]
\begin{bmatrix} c_+\\c_-\end{bmatrix}.
    \end{equation}
    Theorem~2.2 in \cite{crimaldi} yields \eqref{eq:3}. Besides, 
    \begin{equation}
	\label{eq:5}
	\sqrt{T}\frac{M^\pm_T}{Q^\pm_T}=\frac{T}{Q^\pm_T}\times \frac{M^\pm_T}{\sqrt{T}}.
    \end{equation}
    If a sequence $(X_n)_n$ converges $\cF_\infty$-stably and a sequence $(Y_n)_n$ 
    of $\cF_\infty$-measurable
    random variables converges in probability, then $(X_n,Y_n)_n$ converges $\cF_\infty$-stably.
    Using the property in \eqref{eq:5} and \eqref{eq:5bis} yields  \eqref{eq:4}.
\end{proof}


\subsection{The fundamental system}

Along with the characterization through the scale function and the speed
measure, much information on the process 
can be read from the so-called \emph{fundamental system}
\cite{feller,ito,rogers2}:
For any $\lambda>0$, there exists some functions $\phi_\lambda$ and $\psi_\lambda$ such that 
\begin{itemize}[noitemsep,topsep=0pt]
    \item $\psi_\lambda$ and $\phi_\lambda$ are continuous, positive from $\RR$ to $\RR$ with $\phi_\lambda(0)=\psi_\lambda(0)=1$.
    \item $\psi_\lambda$ is increasing with $\lim_{x\to-\infty}\psi_\lambda(x)=0$, $\lim_{x\to\infty}\psi_\lambda(x)=+\infty$.
    \item $\phi_\lambda$ is decreasing with $\lim_{x\to-\infty}\phi_\lambda(x)=+\infty$, $\lim_{x\to\infty}\phi_\lambda(x)=0$.
    \item $\phi_\lambda$ and $\psi_\lambda$ are solutions to $\cL f=\lambda f$.
\end{itemize}

In the case of piecewise constant coefficients with one discontinuity at $0$, 
these solutions may be computed as linear
combinations of the minimal functions for constant coefficients. Using 
the fact that $\phi$, $\phi'$, $\psi$ and $\psi'$ are continuous at $0$, 
\begin{align}
    \label{increasingsolution}
    \psi_\lambda(x)&=\begin{cases}
    \exp\left(x\frac{-b_-+\sqrt{b_-^2+2\sigma_-^2\lambda}}{\sigma_-^2}\right)&\text{ if }x<0\\
    \kappa_+ \exp\left(x\frac{-b_++\sqrt{b_+^2+2\sigma_+^2\lambda}}{\sigma_+^2}\right)
    + \delta_+ \exp\left(x\frac{-b_+-\sqrt{b_+^2+2\sigma_+^2\lambda}}{\sigma_+^2}\right)&\text{ if }x\geq0,
\end{cases}
    \\
    \label{decreasingsolution}
    \phi_\lambda(x)&=
    \begin{cases}
	\kappa_-\exp\left(x\frac{-b_--\sqrt{b_-^2+2\sigma_-^2\lambda}}{\sigma_-^2}\right)
	+\delta_-\exp\left(x\frac{-b_-+\sqrt{b_-^2+2\sigma_-^2\lambda}}{\sigma_-^2}\right)&\text{ if }x<0,\\
	\exp\left(x\frac{-b_+-\sqrt{b_+^2+2\sigma_+^2\lambda}}{\sigma_+^2}\right)&\text{ if }x\geq0
    \end{cases}
\end{align}
with 
\begin{align*}
    \kappa_+&\eqdef
\frac{-b_-\sigma_+^2 + b_+\sigma_-^2 + \sigma_-^2\sqrt{b_+^2 + 2\lambda\sigma_+^2} + \sigma_+^2\sqrt{b_-^2 + 2\lambda\sigma_-^2}}{2\sigma_-^2\sqrt{ b_+^2 + 2\lambda\sigma_+^2 }},\\
\delta_+&\eqdef
\frac{b_-\sigma_+^2 - b_+\sigma_-^2 + \sigma_-^2\sqrt{b_+^2 + 2\lambda\sigma_+^2} - \sigma_+^2\sqrt{b_-^2 + 2\lambda\sigma_-^2}}{2\sigma_-^2\sqrt{b_+^2 + 2\lambda\sigma_+^2}},\\
\kappa_-&\eqdef
\frac{-b_-\sigma_+^2 + b_+\sigma_-^2 - \sigma_-^2\sqrt{b_+^2 + 2\lambda\sigma_+^2} + \sigma_+^2\sqrt{b_-^2 + 2\lambda\sigma_-^2}}{2\sigma_+^2\sqrt{b_-^2 + 2\lambda\sigma_-^2}},\\
\delta_-&\eqdef
\frac{b_-\sigma_+^2 - b_+\sigma_-^2 + \sigma_-^2\sqrt{b_+^2 + 2\lambda\sigma_+^2} + \sigma_+^2\sqrt{ b_-^2 + 2\lambda\sigma_-^2 }}{2\sigma_+^2\sqrt{b_-^2 + 2\lambda\sigma_-^2}}.
\end{align*}

We also define the quantities \cite{py}
\begin{align*}
    \widehat{\psi}(\lambda)&\eqdef\frac{1}{2}\frac{\psi'_\lambda(0)}{\psi_\lambda(0)}
    =\frac{-b_-+\sqrt{b_-^2+2\sigma_-^2\lambda}}{2\sigma_-^2}\geq 0\\
    \text{ and }
    \widehat{\phi}(\lambda)&\eqdef-\frac{1}{2}\frac{\phi'_\lambda(0)}{\phi_\lambda(0)}
    =
\frac{b_++\sqrt{b_+^2+2\sigma_+^2\lambda}}{2\sigma_+^2}\geq 0.
\end{align*}
In particular,
\begin{equation*}
    \widehat{\psi}(0)=0\text{ and }\widehat{\phi}(0)=\frac{b_+}{\sigma_+^2}
    \text{ when $b_-\geq 0$ and $b_+\geq 0$}.
\end{equation*}


\subsection{Last passage time and occupation time for the transient process}
\label{sub:transient}

When $\xi$ is a transient process, the last passage time 
$\ell_0=\sup\Set{t\geq 0\given \xi_t=0}$
of~$\xi$ at~$0$ is almost surely finite. 
Its Laplace transform is (See (53) in \cite{py}): 
\begin{equation*}
    \EE_0[\exp(-\lambda \ell_0)]
    =
    \frac{\widehat{\psi}(0)+\widehat{\phi}(0)}{
    \widehat{\psi}(\lambda)+\widehat{\phi}(\lambda)}.
\end{equation*}
Let us now assume that $b_+>0$ and $b_-\geq 0$. This is the 
transient case \textbf{T0} where the process ends up almost surely in the positive semi-axis.
Thus, $Q^-_{\ell_0}=Q^-_{\infty}$ and $L_{\ell_0}(\xi)=L_{\infty}(\xi)$.

Let us write $\widehat{b}_\pm\eqdef b_\pm/\sigma_\pm^2$.
From Corollary~5 in \cite{py},
\begin{equation*}
    \EE_0[\exp(-\alpha L_\infty(\xi)-\lambda Q_{\infty}^-)]
    =
    \frac{\widehat{\psi}(0)+\widehat{\phi}(0)}{%
	\alpha+
    \widehat{\phi}(0)+\widehat{\psi}(\lambda)}
    =\frac{\widehat{b}_+}{\alpha+\widehat{b}_+ -\frac{\widehat{b}_-}{2}
+\frac{1}{2\sigma_-^2}\sqrt{b_-^2+2\sigma_-^2\lambda}}.
\end{equation*}
Let $p_{L_\infty(\xi)}(t)$ be the density of $L_\infty(\xi)$.
Setting $\lambda=0$, we see that $L_\infty(\xi)$ is distributed according to an exponential distribution
of rate $\widehat{b}_+$. Thus, $p_{L_\infty(\xi)}(t)=\widehat{b}_+\exp(-\widehat{b}_+t)$. 
A conditioning shows that 
\begin{equation*}
    \EE_0[\exp(-\alpha L_\infty(\xi)-\lambda Q_{\infty}^-)]
=\int_{0}^{+\infty} \exp(-\alpha t)
\EE_0\Prb*{\exp(-\lambda Q_{\infty}^-)\given L_\infty(\xi)=t}
    p_{L_{\infty}}(t)\vd t.
\end{equation*}

By inverting the Laplace transform with respect to $\alpha$, 
since $p_{L_{\infty}}(t)=\widehat{b}_+\exp(-\widehat{b}_+t)$, 
\begin{equation*}
    \EE_0\Prb*{\exp(-\lambda Q_{\infty}^-)\given L_\infty(\xi)=t}
    =
    \exp\left(
	-t\left(
	    -\frac{\widehat{b}_-}{2}
	    +\frac{1}{\sqrt{2}\sigma_-}\sqrt{\frac{b_-^2}{2\sigma_-^2}+\lambda}
	\right)
    \right).
\end{equation*}
Inverting the latter Laplace transform with respect to $\lambda$, 
the density $p_{Q_{\infty}^-}(s|t)$ of $Q^-_{\infty}$ given $\Set{L_\infty(\xi)=t}$ is 

\begin{equation*}
	p_{Q_{\infty}^-}(s|t)
    =
    \frac{t}{\sigma_-2\sqrt{2\pi}s^{3/2}}
    \exp\left(
	\frac{tb_-}{2\sigma_-^2}
	-
	\frac{b_-^2 s}{2\sigma_-^2}
	-
	\frac{t^2}{8 \sigma_-^2 s}
    \right).
\end{equation*}

Hence, the distribution $p_{(Q^-_{\infty},L_\infty(\xi))}(s,t)$ of the pair $(Q^-_{\infty},L_\infty(\xi))$ is
\begin{equation*}
    p_{(Q^-_{\infty},L_\infty(\xi))}(s,t)
    =
    \frac{t b_+}{\sigma_+^2\sigma_-2\sqrt{2\pi}s^{3/2}}
    \exp\left(
	\left(\frac{b_-}{2\sigma_-^2}-\frac{b_+}{\sigma_+^2}\right)t
	-
	\frac{b_-^2 s}{2\sigma_-^2}
	-
	\frac{t^2}{8 \sigma_-^2 s}
    \right).
\end{equation*}

The density $p_{\cR_{\mathbf{T0}}}(r)$ of the random variable 
$\cR_{\mathbf{T0}}\eqdef L_\infty(\xi)/2Q^-_{\infty}$
is then  
\begin{multline*}
    p_{\cR_{\mathbf{T0}}}(r)=2\int_0^{+\infty} s \cdot p_{(Q^-_{\infty},L_\infty(\xi))}\left(s,2rs\right)\vd s
\\
=
\frac{rb_+}{\sigma_+^2\sigma_-\sqrt{2}} \bigg(\frac{2rb_+}{\sigma_+^2}+\frac{(r-b_-)^2}{2\sigma_-^2}\bigg)^{-3/2},\ r>0.
\end{multline*}

The distribution of $L_\infty(\xi)/2Q^+_{\infty}$ when $b_-<0$, $b_+\leq 0$ is found by symmetric arguments.

Let us now assume that $b_+>0$ and $b_-< 0$. This is the 
transient case \textbf{T1} where the process can end up in both semi-axis. 
The Laplace transform is 
\begin{equation*}
    \EE_0[\exp(-\alpha L_\infty(\xi)-\lambda Q_{\infty}^-)]
    =
    \frac{\widehat{\psi}(0)+\widehat{\phi}(0)}{%
	\alpha+
    \widehat{\phi}(0)+\widehat{\psi}(\lambda)}
    =\frac{\widehat{b}_+ - \widehat{b}_-}{\alpha+\widehat{b}_+ -\frac{\widehat{b}_-}{2}
+\frac{1}{2\sigma_-^2}\sqrt{b_-^2+2\sigma_-^2\lambda}}.
\end{equation*}
With analogous computations as before we get that the density $p_{\cR^-_{\mathbf{T1}}}(r)$ 
    of $\cR^-_{\mathbf{T1}}$ is
\begin{equation*}
    p_{\cR^-_{\mathbf{T1}}}(r)
=\frac{r}{\sigma_-\sqrt{2}} \bigg(\frac{b_+}{\sigma_+^2}-\frac{b_-}{\sigma_-^2}\bigg)
\bigg(\frac{2rb_+}{\sigma_+^2}+\frac{(r-b_-)^2}{2\sigma_-^2}\bigg)^{-3/2},\ r>0.
\end{equation*}
Considering also the previous case, we can write the following formula, holding
for the density of $\cR=\cR_{\mathbf{T0}}$ or $\cR=\cR_{\mathbf{T1}}^-$ in both
cases \textbf{T0} and \textbf{T1}:
\begin{equation}
    \label{eq:dens:r}
    p_{\cR}(r)
=\frac{r}{\sigma_-\sqrt{2}}\bigg(\frac{b_+}{\sigma_+^2}+\frac{\mpart{b_-}}{\sigma_-^2}\bigg)
 \bigg(\frac{2rb_+}{\sigma_+^2}+\frac{(r-b_-)^2}{2\sigma_-^2}\bigg)^{-3/2},\ r>0.
\end{equation}
Notice that this is the density of a positive random variable which is not
integrable. This gives the limit behavior of the estimator $\beta^-_T$ of
$b_-$. The behavior of $\beta_T^+$ in the corresponding cases can be found by
symmetric arguments.

\section{Proofs of the asymptotic behavior of the estimators}

\label{sec:proofs}


\subsection{Asymptotic behavior for the ergodic case  (\textbf{E})}

\label{sec:conv:ergodic}

The ergodic case is the most favorable one. The process $\xi$ is ergodic, 
so that for any bounded, measurable function $f$, 
$\frac{1}{T}\int_0^T f(\xi_s)\vd s$ converges almost surely to $\int f(x)\frac{m(x)}{M(\RR)}\vd x$. 

With the explicit expression of $M$ that follows from \eqref{eq:speend-n-scale},  
\begin{equation*}
    M(\RR_+)=\frac{-1}{b_+},\ M(\RR_-)=\frac{1}{b_-}
    \text{ and }M(\RR)=\frac{\abs{b_+b_-}}{\abs{b_-}+\abs{b_+}}.
\end{equation*}
From the ergodic theorem, since $Q^\pm_t=\int_0^t \ind{\pm\xi_s\geq 0}\vd s$, 
\begin{equation*}
    \frac{Q^\pm_T}{T}\convas[T\to\infty]\frac{M(\RR_\pm)}{M(\RR)}
\end{equation*}
so that 
\begin{equation}
    \label{eq:E:1}
\frac{Q^+_T}{T}\convas[T\to\infty]\frac{\abs{b_-}}{\abs{b_-}+\abs{b_+}}
\text{ and }
\frac{Q^-_T}{T}\convas[T\to\infty]\frac{\abs{b_+}}{\abs{b_-}+\abs{b_+}}.
\end{equation}

The rate of convergence follows from Proposition~\ref{prop:clt} and \eqref{eq:E:1}.

\subsection{Asymptotic behavior for the null recurrent case with vanishing drift (\textbf{N0})}

\label{sec:N0}

When $b_-=b_+=0$, the process $\xi$ is an Oscillating Brownian motion
(OBM, introduced first in \cite{kw}, see also \cite{LP}).
Supposing $\xi_0=0$, using the scaling relation \cite[Remark 3.7]{LP}, for any $T>0$, 
\begin{equation*}
    \left(\frac{\ppart{\xi_T}}{\sqrt{T}},\frac{\mpart{\xi_T}}{\sqrt{T}},
    \frac{L_T(\xi)}{\sqrt{T}},\frac{Q^+_T}{T}\right)\eqlaw
    (\ppart{\xi_1},\mpart{\xi_1},L_1(\xi),Q^+_1).
\end{equation*}
Therefore, 
\begin{equation*}
\sqrt{T}
\begin{pmatrix}
\beta_T^+ \\
\beta_T^-
\end{pmatrix}
=
\begin{pmatrix}
    \dfrac{\ppart{\xi_T}/\sqrt{T} - \frac{1}{2} L_T(\xi)/\sqrt{T}
}{Q^+_T/T} \\
\dfrac{\frac{1}{2} L_T(\xi)/\sqrt{T}
    -\mpart{\xi_T}/\sqrt{T}
}{Q^-_T/T}
\end{pmatrix}
\eqlaw
\begin{pmatrix}
\beta^+_1 \\
\beta^-_1
\end{pmatrix}
=
\begin{pmatrix}
    \dfrac{\ppart{\xi_1}- \frac{1}{2} L_1(\xi)}{Q^+_1} \\
    \dfrac{\frac{1}{2} L_1(\xi)-\mpart{\xi_1}}{Q^-_1}
\end{pmatrix}.
\end{equation*}
We recall now that $X=\Phi(\xi)\eqdef \xi/\sigma(\xi)$ is a Skew Brownian motion~\cite{LP,etore2}.
An explicit formula for the density
for the position a Skew Brownian motion, its local time and
its occupation time 
is known \cite{abtww,gairat}. Since the transform~$\Phi$ is piecewise
linear, one easily recover the one of an OBM, its local and occupation times.
Hence, the density of $(\xi_1,L_1(\xi),Q^+_1(\xi))$ is 
\begin{multline}
p_{(\xi_1,L_1(\xi),Q^+_1)}(\rho,\lambda,\tau)\\
=
\begin{cases}
\dfrac{(\lambda/2+\rho)\lambda/2 }{2\pi \sigma_- \sigma_+^3 (1-\tau)^{3/2}\tau^{3/2}} 
\exp\left(-\dfrac{(\lambda/2)^2}{2\sigma_-^2(1-\tau)} -\dfrac{(\lambda/2+\rho)^2}{2\sigma_+^2\tau}\right)
&\text{ for }\rho\geq 0,\\[15pt]
\dfrac{(\lambda/2-\rho)\lambda/2 }{2\pi \sigma_+ \sigma_-^3 (1-\tau)^{3/2}\tau^{3/2}} 
\exp\left(-\dfrac{(\lambda/2)^2}{2\sigma_+^2\tau} -\dfrac{(\lambda/2-\rho)^2}{2\sigma_-^2(1-\tau)}\right)&\text{ for }\rho<0.
\end{cases}
\end{multline}
The change of variable in the density suggested by
\begin{equation*}
    (\beta^+_1,\beta^-_1,Q^+_1)=\left( \frac{\ppart{\xi_1}-L_1(\xi)/2}{Q^+_1},\frac{L_1(\xi)/2-\mpart{\xi_1}}{1-Q^+_1},Q^+_1\right)
\end{equation*}
gives 
\begin{multline*}
p_{(\beta^+_1,\beta^-_1,Q^+_1)}(a,b,\delta)\\
=
2\delta(1-\delta) p_{(\xi_1,L_1(\xi),Q^+_1)}(a\delta+b(1-\delta),
|a\delta+b(1-\delta)|-a\delta+b(1-\delta)
,\delta)
\end{multline*}
and then, since $Q^+_1\in[0,1]$,  
\begin{multline}
    \label{explicitdensitynodrift}
p_{(\beta^+_1,\beta^-_1)}(a,b)\\
= 
\int_0^1
2\delta(1-\delta) p_{(\xi_1,L_1(\xi),Q^+_1)}
(a\delta+b(1-\delta),
|a\delta+b(1-\delta)|-a\delta+b(1-\delta)
,\delta) \vd\delta.
\end{multline}

\subsection{Asymptotic behavior for the null recurrent case with non-vanishing drift (\textbf{N1})}
\label{sec:conv:N1}

We consider $b_+=0$, $b_->0$. The particle is then pushed upward when its position
is negative. Yet the process is only null recurrent.
The measure $M$ satisfies
\begin{equation}
\label{eq:M:1}
M(\RR_-)=\frac{1}{b_-}\text{ and }M(\RR_+)=+\infty.
\end{equation}

Using 9) in \cite[Section 6.8, p.~228]{ito} or \cite{watanabe58,maruyama},
with \eqref{eq:M:1}, 
\begin{equation*}
    \frac{Q^-_T}{T}\convas[T\to\infty]\frac{M(\RR_-)}{M(\RR)}=0.
\end{equation*}
Since $Q_T^++Q_T^-=T$, it holds that $Q^+_T/T$ converges almost surely to $1$.

Using Proposition~\ref{prop:clt} on $M^+$ only, we obtain 
that 
\begin{equation}
    \label{eq:4bis}
    \frac{M_T^+}{\sqrt{T}} \convl[T\to\infty]\sigma^+\cN^+
\end{equation}
for a Gaussian random variable $\cN^+\sim\cN(0,1)$ and then that  
$\sqrt{T}(\beta^+_T-b_+)\convl[T\to\infty] \sigma^+\cN^+$.

The asymptotic behavior of $Q^-_T$ is more delicate to deal with 
as the process $\xi$ is only null recurrent. For this, we
use the results of \cite{hopfner} which extends the one 
of D.A. Darling and M.~Kac \cite{dk} on additive and
martingale additive functionals.

The Green kernel with respect to the invariant measure $M$ of $\cL$ is given by~\cite{feller,ito,rogers2}
\begin{equation*}
    g_\lambda(x,y)\eqdef\frac{1}{W_\lambda}\begin{cases}
	\psi_\lambda(x)\phi_\lambda(y)&\text{ if }x<y,\\
	\phi_\lambda(x)\psi_\lambda(y)&\text{ if }x\geq y
    \end{cases}
    \text{ with }
    W_\lambda=\frac{\psi_\lambda'(0)\phi_\lambda(0)-\psi_\lambda(0)\phi_\lambda'(0)}{S'(0)}.
\end{equation*}
Using \eqref{increasingsolution}-\eqref{decreasingsolution}, 
\begin{equation}
    \label{eq:wronsk:1}
W_\lambda=
 \frac{\sqrt{2\lambda}}{\sigma_+} 
+ \frac{\sqrt{b_-^2 + 2\lambda\sigma_-^2}}{\sigma_-^2}
-\frac{b_-}{\sigma_-^2}\xrightarrow[\lambda\to 0]{} \frac{\sqrt{2}}{\sigma_+}.
\end{equation}

On the other hand, it follows
from \eqref{increasingsolution}-\eqref{decreasingsolution} that 
\begin{equation}
    \label{eq:phipsi:1}
    \psi_\lambda(x)\xrightarrow[\lambda\to 0]{}\psi_0(x)\eqdef 1
    \text{ and }
    \phi_\lambda(x)\xrightarrow[\lambda\to 0]{}\phi_0(x)\eqdef 1
    \text{ for any }x\in\RR.
\end{equation}

For a measurable function $f:\RR\to\RR_+$ such 
that $\int_\RR f\vd M<+\infty$, combining \eqref{eq:wronsk:1} and \eqref{eq:phipsi:1} yields
that 
\begin{equation*}
    \sqrt{\lambda}   \int_{\RR} g_\lambda(x,y)f(x)m(y)\vd y\xrightarrow[\lambda\to 0]{}
    \frac{\sigma_+}{\sqrt{2}}\int_\RR f(y) m(y)\vd y,\ \forall x\in\RR.
\end{equation*}
The above convergence allows one to identify the parameters to use in the limit theorem
we will use on $(M^-,Q^-)$.
We then define $\ell(\lambda)\eqdef \sqrt{2}/\sigma_+$ which is a constant function, 
and $\alpha\eqdef 1/2$, the exponent of $\lambda$.

Let $(\cM_t)_{t\geq 0}$ be a Mittag-Leffler process of index $\alpha=1/2$ (it is the inverse
of an increasing stable process of index $1/2$). The process $2^{-1/2}\cM$
is equal in distribution to the running maximum of a Brownian motion, 
or equivalently, to the local time of a Brownian motion \cite[Remark~2.9, p.~21]{hopfner}. 

From Theorem~3.1 and Corollary~3.2 in \cite[p.~26]{hopfner}, 
since $\bra{M^-}_t=\sigma^2_- Q^-_t$, $t\geq 0$ and $M$ is continuous, 
\begin{equation}
    \label{eq:6}
    \left(\sqrt{\ell(n)}\frac{M^{-}_{nt}}{n^{1/4}},\ell(n)\frac{Q^{-}_{nt}}{n^{1/2}}
    \right)_{t\in [0,1]}
    \convl[n\to\infty] 
    \left(\sigma_-\sqrt{\nu} B^-(\cM_t),\nu \cM_t
    \right)_{t\in [0,1]}
\end{equation}
with respect to the uniform topology, where $B$ is a Brownian
motion independent from $\cM$ and 
\begin{equation*}
    \nu=\int_{\RR} m(y)\EE_y\Prb*{\int_0^1 \ind{\xi_s\leq 0}\vd s}\vd y
    =M(\RR_-)=\frac{1}{b_-}.
\end{equation*}

From the reflection principle, the distribution of $\cM_1$
is the same as the one of a truncated normal distribution $\sqrt{2}\cT$ where
$\cT\eqdef \abs{\cG}$ with $\cG\sim\cN(0,1)$.
Setting $t=1$ in \eqref{eq:6} and using the scaling property of the Brownian 
motion $B^-$, 
\begin{equation*}
    \left(\frac{M^-_T}{T^{1/4}},\frac{Q^{-}_T}{T^{1/2}}\right)
    \convl[T\to\infty]
	\left(
	\sigma_-\frac{\sqrt{\sigma_+}}{\sqrt{b_-}}\sqrt{\cT}\cdot \cN^-,
	\frac{\sigma_+}{b_-}\cT
    \right)
\end{equation*}
for a Gaussian random variable $\cN^-\sim\cN(0,1)$ independent from $\cT$.

It remains to show the independence of $\cN^+$, $\cN^-$ and $\cT$. 

Since $\bra{M^\pm}_t=\sigma_\pm^2Q^\pm_t$
and $\bra{M^+,M^-}_t=0$ for any $t\geq 0$, the Knight theorem
\cite[Theorem~7.3', p.~92]{ikeda}
implies that there exists on an extension of $(\Omega,\cF,\PP)$ 
a $2$-dimensional Brownian motion $(B^+,B^-)$
such that $M^\pm_t=B^\pm(\sigma_\pm^2Q^\pm_t)$ for any $t\geq 0$.
Let us set $B^+_n(t)=n^{-1/2}B^+(nt)$ and $B^-_n(t)=n^{-1/4}B^-(\sqrt{n}t)$
for any integer $n$ and any $t\geq 0$. From the scaling property, 
$(B^+_n,B^-_n)$ is still a $2$-dimensional Brownian motion 
which converges in distribution to a $2$-dimensional Brownian 
motion $(B^+_\infty,B^-_\infty)$ in the space $\cC([0,1],\RR^2)$ of continuous functions.

For any $0\leq s\leq t \leq 1$,
$Q^+_{nt}-Q^+_{ns}\leq n(t-s)$ so that $\Set{(Q^+_{nt}/n)_{t\in[0,1]}}_{n\geq 1}$ is 
tight in the space of continuous functions.
Hence, $\Set*{(Q^+_{nt}/{n})_{t\in[0,1]}}_{n\geq 1}$ converges in probability 
to the identity map $t\mapsto t$ in the space of continuous function $\cC([0,1],\RR)$.

Combining this result with~\eqref{eq:6}, it holds that 
$\Set*{\left(B^+_n(t),B^-_n(t),n^{-1}Q^+_{nt},n^{-1/2}Q^-_{nt}\right)_{t\in[0,1]}}_{n\geq 1}$
is tight in $\cC([0,1],\RR^4)$ and then necessarily
\begin{equation}
    \label{eq:conv:1}
    \left(B^+_n(t),B^-_n(t),\frac{Q^+_{nt}}{n},\frac{Q^-_{nt}}{\sqrt{n}}\right)_{t\in[0,1]}
    \convl[n\to\infty]{}
    (B^+_\infty(t),B^-_\infty(t),t,\nu \cM_t)_{t\in[0,1]}
\end{equation}
in the space of continuous functions $\cC([0,1],\RR^4)$.
Being the inverse of a $1/2$-stable process, hence a pure jump process, 
$\cM$ is independent from $(B^+_\infty,B^-_\infty)$
for the arguments presented in \cite[p.~38]{hopfner} or \cite[Theorem~6.3, p.~77]{ikeda}.

For any $t\in[0,1]$, owing to the definition of $B^+_n$ and $B^-_n$, 
\begin{align*}
    \frac{M^+_{nt}}{\sqrt{n}}
    &=\frac{1}{n^{1/2}}B^+(\sigma_+^2 Q^+_{nt})
    = B^+_n\left(\sigma_+^2\frac{Q^+_{nt}}{n}\right)
    \\
    \text{ and }
    \frac{M^-_{nt}}{n^{1/4}}
    &=\frac{1}{\sqrt{n}}B^-(\sigma_-^2Q^-_{nt})
    = B^-_n\left(\sigma_-^2\frac{Q^-_{nt}}{\sqrt{n}}\right).
\end{align*}
Using $(n^{-1}Q^+_{nt})_{t\in[0,1]}$
and $(n^{-1/2}Q^-_{nt})_{t\in[0,1]}$ as random time changes, we deduce from~\eqref{eq:conv:1} and the results 
in \cite[p.~144]{billingsley} that 
\begin{equation*}
    \left(\frac{M^+_{nt}}{n^{1/2}},\frac{M^-_{nt}}{n^{1/4}},
    \frac{Q^+_{nt}}{n},\frac{Q^-_{nt}}{\sqrt{n}}
    \right)_{t\in [0,1]}
    \convl[ n\to\infty ]{}
    (B_\infty^+(\sigma_+^2 t),B_\infty^-(\sigma_-^2\nu\cM_t),t,\nu\cM_t)_{t\in[0,1]}.
\end{equation*}
By setting $\cN^+\eqdef B_\infty^+(\sigma_+^2)/\sigma_+\sim\cN(0,1)$ in \eqref{eq:4bis}, 
$\cT\eqdef \cM_1/\sqrt{2}$ and $\cN^-\eqdef B_\infty^-(\sigma_-^2\nu)/\sigma_-\sqrt{\nu}\sim\cN(0,1)$, 
this proves \eqref{eq:N1} using $t=1$ in the above limit.

\subsection{Asymptotic behavior for the transient case (\textbf{T0})}
\label{sec:conv:T0}

We recall that we consider only $b_+>0$ and $b_-\geq 0$.

It is known that the last passage time $\ell_0$ to $0$ is almost surely 
finite so that $Q^-_\infty<\infty$ almost surely. Since $Q^+_T=T-Q^-_T$, 
we obtain that $Q^+_T/T$ converges almost surely to $1$.
The convergence results regarding $\beta^+_T$ follows from 
Proposition~\ref{prop:clt} applied only to one component.

The asymptotic behavior of $\beta_T^-$ follows from Proposition~\ref{prop:convergence}(ii).

When $T>\ell_0$ and $\xi_0=0$, then $\xi_T>0$ and thus 
$\beta_T^-=L_T(\xi)/2Q^-_T$. Yet the local time $L_T(\xi)$
and the occupation times $Q^-_T$ are constant when $T>\ell_0$. The result follows 
by the computations of Section~\ref{sub:transient}. 

\subsection{Asymptotic behavior for the transient case generated by diverging drift (\textbf{T1})}
\label{sec:conv:T1}

When $b_-<0$ and $b_+>0$, the process is also transient as $S(+\infty)<+\infty$ and $S(-\infty)>-\infty$. 
The scale function $S$ map $\RR$ to $(-\gamma_-,\gamma_+)$ with 
$\gamma_\pm=\abs{\sigma_\pm^2}{2b_\pm}$. From the Feller test \cite[Theorem~5.29]{ks}, the process does not explode. 
Thus, as $\xi_0=0$, it follows from \cite{watanabe} or \cite[Proposition 5.5.22, p.~354]{ks}
that 
\begin{equation*}
    p\eqdef \PP\Prb*{\lim_{T\to\infty} \xi_T=+\infty}=1-\PP\Prb*{\lim_{T\to\infty} \xi_T=-\infty}
    =\frac{\gamma_-}{\gamma_-+\gamma_+}
    =\frac{\sigma_-^2b_+}{\sigma_+^2b_-+\sigma_-^2b_+}.
\end{equation*}

Then event that $\Set{\lim_{T\to\infty} \abs{\xi_T}=+\infty}$ arise
when the process starts an excursion with infinite lifetime, thus 
after the last passage time to $0$.
We denote by $A^\pm$ the event $\Set{\lim_{T\to\infty}Q^\pm_T/T=1}$,
so that $A^+\cap A^-=\emptyset$. 
Hence, $\PP\Prb{A^+}=1-\PP\Prb{A^-}=p$.

Using the same arguments as in the case \textbf{T0}, given $A^\pm$, 
$\beta_T^\pm$ converges almost surely to $b_\pm$ while $\beta_T^\mp=\pm L_\infty(\xi)/Q^\mp_{\ell_0}$.

For the Central Limit Theorem, we apply Corollary~2.3 in \cite{crimaldi}
on $M^+_T/\sqrt{T}$. As $Q_T^+>0$ a.s. as soon as $T>0$ since $\xi_0=0$ and 
\begin{equation*}
    \Set{\cB=p}=A^+,\ \text{a.s. for }p=0,1,
\end{equation*}
it follows that for a normal distribution $\cN$, 
\begin{equation*}
    \sqrt{\frac{T}{\sigma_+^2 Q_T^+}}\cdot\frac{M_T}{\sqrt{T}}\xrightarrow[N\to\infty]{\cF_\infty\mathrm{-stably}} \cN
    \text{ under }\PP\Prb{\cdot\given \cB=1}.
\end{equation*}
It follows that
\begin{equation*}
    \beta_T^+-b_+=
    \sqrt{T}\frac{M_T}{Q^+_T}=\sigma_+\sqrt{\frac{T}{Q^+_T}}\sqrt{\frac{T}{\sigma_+^2 Q^+_T}}
    \frac{M_T}{\sqrt{T}}
\xrightarrow[N\to\infty]{\cF_\infty\mathrm{-stably}} \sigma_+ \cN
    \text{ under }\PP\Prb{\cdot\given \cB=1}.
\end{equation*}
Hence the result.


\section{Wilk's theorem and LAN property}

\label{sec:loglik}

Owing to the quadratic nature of the log-likelihood, 
we easily deduce both a Wilk theorem, on which a 
hypothesis test may be developed, as well as the
Local Asymptotic Normality (LAN) property, which 
proves that our estimators are asymptotically efficient.

\subsection{Wilk's theorem and a hypothesis test}
\label{sec:wilk}

The log-likelihood $\log G_T(b_+,b_-)$ can be computed from the data
using \eqref{eq:girsanov}. Moreover, this function is quadratic in $b_+$ and $b_-$.  
We have then 
\begin{align*}
    \nabla \log G_T(b_+ ,b_-)
    &=
    \begin{bmatrix}
    \frac{\ppart{\xi_T}}{\sigma_+^2}-\frac{\ppart{\xi_0}}{\sigma_+^2}
    -\frac{1}{2\sigma_+^2}L_T(\xi)-\frac{b_+}{\sigma_+^2}Q_T^+
    \\
    -\frac{\mpart{\xi_T}}{\sigma_-^2}+\frac{\mpart{\xi_0}}{\sigma_-^2}
    +\frac{1}{2\sigma_-^2}L_T(\xi)-\frac{b_-}{\sigma_-^2}Q_T^-
\end{bmatrix}
\\
\text{ and }
     \Hess\log G_T(b_+ ,b_-)
     &=
    \begin{bmatrix}
	-\frac{Q_T^+}{\sigma_+^2} & 0 \\
	0 & -\frac{Q_T^-}{\sigma_-^2}
    \end{bmatrix}.
\end{align*}
Therefore, around any point $(b_+,b_-)$, 
\begin{multline}
    \label{eq:hyp:3}
\log G_T(b_++\Delta b_+,b_-+\Delta b_-)
    =
    \log G_T(b_+,b_-)
\\
+
\nabla \log G_T(b_+,b_-)\cdot \begin{bmatrix} 
    \Delta b_+\\\Delta b_-
\end{bmatrix}
-\frac{Q_T^+}{2\sigma_+^2} \Delta b_+^2
-\frac{Q_T^-}{2\sigma_-^2} \Delta b_-^2.
\end{multline}

In particular, we prove a result in the asymptotic behavior of the log-likelihood
for the ergodic case \textbf{E} or the null recurrent case \textbf{N1}. 
A similar result can be given for the null recurrent cases \textbf{N0}
with a different limit distribution that can be identified with the density 
given in Section~\ref{sec:N0}.

\begin{proposition}[Wilk's theorem; Ergodic case \textbf{E} or null recurrent case \textbf{N1}]
    \label{prop:hyp:E}
Denote by $(\btrue_+,\btrue_-)$ be the real parameters. 
Then 
\begin{equation}
    \label{eq:hyp:1}
    -2\log \frac{G_T(\btrue_+,\btrue_-)}{G_T(\beta^+_T,\beta^-_T)}\convl[T\to\infty]{}
    \chi_2
    \eqdef 
    (\cN^+)^2
    +
    (\cN^-)^2
    \text{ under }\PP_{(\btrue_+,\btrue_-)}
\end{equation}
for two independent, unit Gaussian random variables
$\cN^+$ and $\cN^-$.
Besides, when $(b_+,b_-)\not=(\btrue_+,\btrue_-)$, then 
\begin{equation}
    \label{eq:hyp:2}
    -\log \frac{G_T(b_+,b_-)}{G_T(\beta^+_T,\beta^-_T)}\convas[T\to\infty]{} +\infty
\text{ under }\PP_{(\btrue_+,\btrue_-)}.
\end{equation}
\end{proposition}
\begin{proof}
    Considering \eqref{eq:hyp:3} at $(b_+,b_-)=(\beta^+_T,\beta^-_T)$
since $\nabla \log G_T(\beta^+_T,\beta^-_T)=(0,0)$, for any parameter $(b_+,b_-)$
and any $\alpha_+,\alpha_->0$, 
\begin{equation*}
    -2\log G_T(b_+,b_-)/G_T(\beta^+_T,\beta^-_T)=
    \frac{Q_T^+}{T^{\alpha_+}\sigma_+^2} \big(T^{\frac{\alpha_+}{2}}(b_+-\beta^+_T)\big)^2
    +\frac{Q_T^-}{T^{\alpha_-}\sigma_-^2} \big(T^{\frac{\alpha_-}{2}}(b_--\beta^-_T)\big)^2.
\end{equation*}

When the process is ergodic (Case \textbf{E}), 
we set $(\alpha_+,\alpha_-)=(1,1)$. 
It follows from Proposition~\ref{prop:E} that \eqref{eq:hyp:1}
holds. When $(b_+,b_-)\not=(\btrue_+,\btrue_-)$, 
then $\beta^\pm_T-b_\pm$ does not converge to $0$ while 
$Q^+_T$ converges a.s. to infinity. This proves \eqref{eq:hyp:2}.
The result is similar in the case \textbf{N1} with $(\alpha_+,\alpha_-)=(1,1/2)$.
\end{proof}

A \emph{hypothesis test} can be developed from Proposition~\ref{prop:hyp:E}.
The null hypothesis is $(b_+,b_-)=(b_+^0,b_-^0)$ for a given 
drift $(b_+^0,b_-^0)$ while the alternative hypothesis
is $(b_+,b_-)\not=(b_+^0,b_-^0)$.

Using \eqref{eq:hyp:3}, we compute $w=-2\log G_T(b_+^0,b_-^0)/G_T(\beta^+_T,\beta^-_T)$.
The null hypothesis is rejected with a confidence level $\alpha$
if $w>q_\alpha$
where $q_\alpha$ is the $\alpha$-quantile $\PP[\chi_2\leq q_\alpha]=\alpha$
while~$\chi_2$ follows a $\chi^2$ distribution with $2$ degrees of freedom.

By using the quasi-likelihood instead of the likelihood, we obtain in a similar
way for the ergodic case (Case \textbf{E}) that 
\begin{equation*}
    -2(\Lambda_T(\btrue_+,\btrue_-)-\Lambda_T(\beta^+_T,\beta^-_T))
\convl[T\to\infty]{}
    \sigma_+^2(\cN^+)^2+\sigma_-^2(\cN^-)^2
\text{ under }\PP_{(\btrue_+,\btrue_-)},
\end{equation*}
where $\cN^+$ and $\cN^-$ are two unit Gaussian random variables. This could
also be used for an hypothesis test.
When using either the log-likelihood of the quasi-likelihood, $\sigma_+$ and $\sigma_-$
shall be known.

\subsection{The LAN property}

\label{sec:LAN}

The LAN (local asymptotic normal) property, introduced by L.~Le Cam in \cite{lecam53}, 
characterizes the efficiency of the estimator (See also \cite{ibragimov,lecam00}
among many other references). It was extended 
as the LAMN (local asymptotic mixed normal) to deal with 
a mixed normal limits by P.~Jeganathan in \cite{jeganathan82a}. 

The quadratic nature of the log-likelihood as well 
as our limit theorems implies that the LAN (resp. LAMN) 
property is verified in the ergodic case \textbf{E}
(resp. null recurrent case \textbf{N1}).

\begin{proposition}[LAN property; Ergodic case \textbf{E}]
    In the ergodic case \textbf{E}, 
    the LAN property holds for the likelihood at $(\btrue_+,\btrue_-)$
    with rate of convergence $(\sigma_+^2/\sqrt{T},\sigma_-^2/\sqrt{T})$
    and asymptotic Fisher information 
    \begin{equation*}
	\Gamma\eqdef \frac{1}{\abs{b_-}+\abs{b_+}}
    \begin{bmatrix}
	\sigma_+^2\abs{b_-} & 0 \\
	0 & \sigma_-^2\abs{b_+} 
    \end{bmatrix}.
    \end{equation*}
\end{proposition}

\begin{proof}
At $(\btrue_+,\btrue_-)$, the gradient may be written
\begin{equation*}
    D(\btrue_+,\btrue_-)=
    \begin{bmatrix}
	\frac{Q^+_T}{\sqrt{T}\sigma_+^2}\sqrt{T}(\beta_T^+-\btrue_+)
	\\
	\frac{Q^-_T}{\sqrt{T}\sigma_-^2}\sqrt{T}(\beta_T^--\btrue_-)
    \end{bmatrix}.
\end{equation*}

Using \eqref{eq:hyp:3}, 
\begin{multline}
    R(T)\eqdef   \log\frac{G\left(\btrue_++\sigma_+^2\frac{\Delta b_+}{\sqrt{T}},
	\btrue_-+\sigma_-^2\frac{\Delta b_-}{\sqrt{T}}\right)}{
    G_T(\btrue_+,\btrue_-)}\\
    =
    \begin{bmatrix}
    Q^+_T(\beta_T^+-\btrue_+)
	\\
	Q^-_T(\beta_T^--\btrue_-)
    \end{bmatrix}
    \cdot 
    \begin{bmatrix}
	\Delta b_+\\
	\Delta b_-
    \end{bmatrix}
-\frac{\sigma_+^2Q_T^+}{2T} \Delta b_+^2
-\frac{\sigma_-^2Q_T^-}{2T} \Delta b_-^2
\\
=
\frac{1}{\sqrt{T}}
    \begin{bmatrix}
    M^+_T
	\\
	M^-_T
    \end{bmatrix}
    \cdot 
    \begin{bmatrix}
	\Delta b_+\\
	\Delta b_-
    \end{bmatrix}
    -\frac{1}{2T}
    \begin{bmatrix}
	\Delta b_+\\
	\Delta b_-
    \end{bmatrix}
    \cdot 
    \bra{M^+,M^-}_T
    \begin{bmatrix}
	\Delta b_+\\
	\Delta b_-
    \end{bmatrix}.
\end{multline}
With $c=\abs{b_+}+\abs{b_-}$, Proposition~\ref{prop:E} implies
that $(T^{-1/2}M^+,^{-1/2}M^-)$  converges 
in distribution to $\cG\sim\cN(0,\Gamma)$ 
and $T^{-1}\bra{M^+,M^-}_T$ converges to the diagonal, definite positive
matrix $\Gamma$.
This proves the LAN property.
\end{proof}
\begin{proposition}[LAMN property; Null recurrent case \textbf{N1}]
    In the null recurrent case \textbf{N1}, 
    the LAMN property holds for the likelihood at $(\btrue_+,\btrue_-)$
    with rate of convergence $(\sigma_+^2/T^{1/2},\sigma_-^2/T^{1/4})$
    and asymptotic (random) Fisher information 
    \begin{equation*}
	\Gamma\eqdef 
    \begin{bmatrix}
	\sigma_+^2 & 0 \\
	0 & \sigma_-^2\frac{\sigma_+}{b_-}\abs{\cN}
    \end{bmatrix}
    \end{equation*}
    for a normal random variable $\cN\sim\cN(0,1)$.
\end{proposition}


\section{Simulation study}
\label{simulations}

\subsection{From continuous to discrete data}

\label{sec:discrete}

In this section, we apply estimator \eqref{estb} and \eqref{def:b:qml} to simulated processes. We test whether the results are good or not, depending on sign and magnitude of involved quantities. 

Within the framework of Data~\ref{dat:2}, 
$\Set{\xi_t}_{t\in[0,T]}$ is observed on a discrete time grid $\{kT/N; k=0,\dots,N\}$. 
The time step between two observations is $\Delta t\eqdef T/N$.

In \eqref{eq:dis:QR}, we have defined $\mathsf{Q}^\pm_{T,N}$ and $\mathsf{R}^\pm_{T,N}$
which are easily computed from the observations $\Set{\xi_{kT/N}}_{k=0,\dotsc,N}$.

The convergence of $\mathsf{Q}^\pm_{T,N}$  was discussed in \cite{LP} for the OBM
(see also Lemma~\ref{lem:discrete}).
We prove that the speed of convergence is strictly better than $\sqrt{N}$,
meaning that $\sqrt{N}(\mathsf{Q}^+_{T,N} -Q^+_T)\convp[n\rightarrow \infty] 0$.
This result can be extended to the drifted process $\xi$ via the Girsanov
theorem.

Alternatively to computing $\mathsf{R}^\pm_{T,N}$, we may approximate
the local time: 
\begin{itemize}[nolistsep,leftmargin=3em]
    \item The approximation $\mathsf{L}_{T,N}$ to the local time given by \eqref{eq:dis:loctime}
	 is easily implemented.
    \item In \cite{LP,lmt1}, we have also considered two other consistent estimators for the local time. 
For the OBM, they are 
\begin{multline*}
    \mathsf{L}^\dag_{T,N} \\
    \eqdef \frac{-3}{2}\sqrt{\frac{\pi}{2\Delta t}}
    \frac{\sigma_++\sigma_-}{\sigma_+\sigma_-} 
\sum_{k=1}^N    
\big(\ppart{\xi_{kT/N}}-\ppart{\xi_{(k-1)T/N}}\big)
\cdot
\big(\mpart{\xi_{kT/N}}-\mpart{\xi_{(k-1)T/N}}\big)
\end{multline*}
and
\begin{equation*}
    \mathsf{L}^{\times}_{T,N} \\
    \eqdef \frac{4}{\sqrt{2\pi}}\times\frac{1}{\sigma_++\sigma_-}\times\frac{1}{\sqrt{N}}
    \sum_{k=0}^{N-1}\ind{\xi_{kT/N}\xi_{(k+1)T/N+1}<0}.
\end{equation*}
\item The expressions of $\mathsf{L}^\dag_{T,N}$ and $\mathsf{L}^\times_{T,N}$ 
    require $\sigma_+$ and $\sigma_-$, unlike the one of $\mathsf{L}_{T,N}$. 
    To the best of our knowledge, this is the first time that such estimator 
    is considered.
\item  For the Brownian motion, these estimators converge at rate $N^{1/4}$, 
    thanks to the results of \cite{j1}. The proof has not been adapted to OBM 
because of the technical difficulties due to the discontinuity of the
coefficients in $0$.  Anyway we conjecture a rate of $1/4$.  
\item The three different estimators for the local time seem to have all
    comparable accuracy on numerical simulations if the volatility coefficients
    $\sigma_\pm$ are \emph{known}. It looks like estimator $\mathsf{L}^\times_{T,N}$
    is the best one when $\sigma_+ \approx \sigma_-$, whereas
    $\LL_{T,N}$ and $\mathsf{L}^\dag_{T,N} $ look more accurate when $\sigma_+ \neq
    \sigma_-$.

\item Anyway,  if these coefficients are not known, only~$\LL_{T,N}$ can be
    implemented directly, whereas to implement the other estimators we have
    first to estimate~$\sigma_\pm$ from observations of~$\xi$. This problem has
    been thoroughly investigated in~\cite{LP}. As a consequence, a good
    estimation of the local time relies on the estimation of
    $\sigma_\pm$. 
Therefore, we choose to show here the implementation of the
estimator using $\LL_{T,N}$, which does not suffer of this possible
additional problem. 
\end{itemize}

In practice, this means that we are actually showing estimator $\disbeta^\pm_{T,N}$ in \eqref{def:b:qml}
in place of $\beta^\pm_T$.
The numerical results using an approximation of $\beta^\pm_T$ using $\mathsf{Q}^\pm_{T,N}$
and~$\mathsf{L}^\dag_{T,N}$ for the local time are very similar, if $\sigma_\pm$ is supposed to be known.

We do not push further in the present paper the theoretical discussion on the
quality of these discrete time approximations. Some more insights, based on
numerical results, are given in the following section.

\subsection{Implementation and simulation}

The aim of the following section is to show on figures the numerical evidence of the central limit theorems stated in Section \ref{sec:CLT}. We also mean to say something more on the choice of the step of the time grid in relation to the quality of the estimation of the local time. 
The code used for the following simulations have been implemented using the software~\texttt{R}.
We will consider time grids of the type $\{0,T/N,2T/N,\dots,T\}\subset \NN$, with $N\in \NN$, 
and use as approximation of local and occupation time $\LL_{T,N}$ and $\mathsf{Q}^+_{T,N}$.
We actually use $\disbeta^\pm_{T,N}$ instead of $\beta^\pm_{T}$, for large $N$. 
In practice, in some cases
the estimator does not really depend on the local time, but is essentially
determined by the final value of the process and the occupation times. In those
cases, the quality of the discrete time approximation of the local time does
not really matter, and therefore we can take $N$ small. In most cases, however, a
good approximation of the local time is needed in order to observe on
simulations the theoretical central limit behavior expected from Section~\ref{sec:CLT}.
In these cases $N\in\NN$ must be taken large enough.

In what follows, the choice of the parameters is detailed for every figure. The diffusion parameter is taken constant $\sigma_+=\sigma_-=0.01$, the same for all the different simulations. We indicate with $[+]$ and $[-]$ estimation on positive and negative semiaxis, e.g. \textbf{(N1)}$[+]$ stands for \textquote{estimation of $b_+$} in case 
\textbf{N1}.

\begin{figure}[ht!]
    \begin{center}
	\includegraphics[page=1]{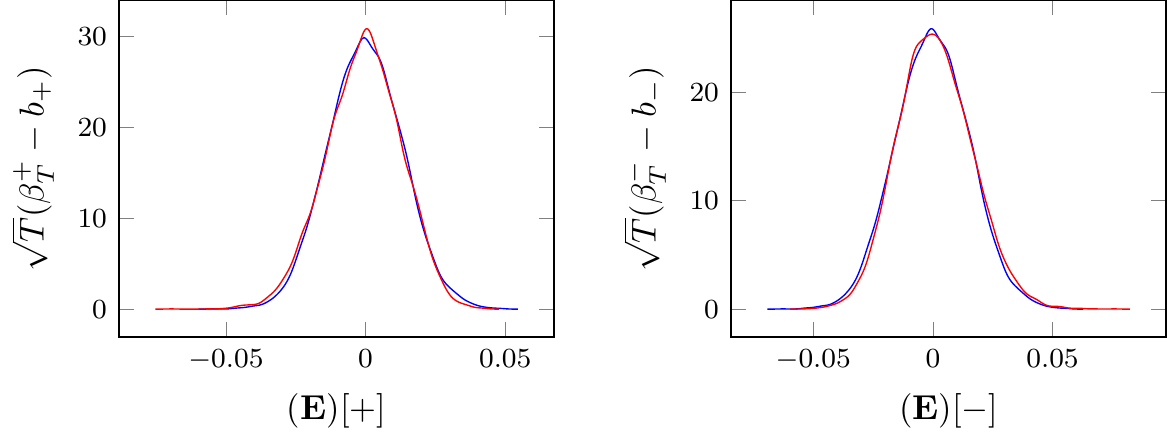}
	\caption{\textbf{(E)}. SDE parameters: $\sigma_\pm=0.01,\, b_-= \num{0.004},\,
    b_+=\num{-0.003}$. Simulation parameters: $T=10^3,\,N=10^5$.
We show both sides (positive and negative) of the estimation, displaying in red the density of 
$\sqrt{T}(\disbeta^\pm_{T,N}-b_\pm)$, in blue its theoretical limit.
The CLT in Proposition \ref{prop:E} is accurate for large~$T$ and time step $T/N$ small, since the quality of the estimation of the local time is key in this case. The limit behavior is Gaussian.}
    \end{center}
\end{figure}

\begin{figure}[ht!]
    \begin{center}
	\includegraphics[page=2]{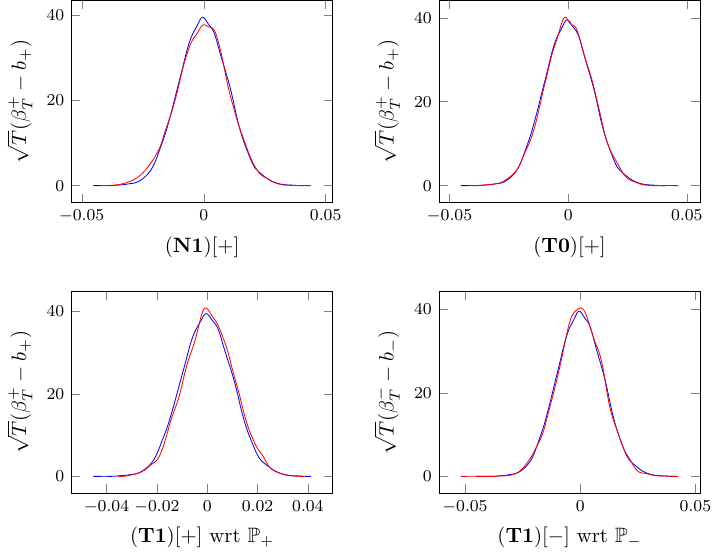}
\caption{
    \textbf{(N1)}$[+]$, \textbf{(T0)}$[+]$, \textbf{(T1)}$[+]$ w.r.t $\PP_+$, \textbf{(T1)}$[-]$ w.r.t $\PP_-$. SDE parameters: $\sigma_\pm=\num{0.01}$;
in case \textbf{N1}: $b_-=\num{0.004},\, b_+=0;$
in case \textbf{T0}: $b_-=\num{0.004},\, b_+=\num{0.006}$;
in case \textbf{T1}: $b_-=\num{-0.004},\, b_+=\num{0.003}$.
Simulation parameters:
$T=10^3,\,N=10^3$.
We display in red the density of $\sqrt{T}(\disbeta_{T,N}^+-b_+)$, in blue its theoretical limit, in cases \textbf{N1} and \textbf{T0}.
We also show in red
$\sqrt{T}(\disbeta_{T,N}^\pm-b_\pm)$ in case~\textbf{T1}, but the density is w.r.t $\PP_\pm$ (cf. \eqref{prop:T1}). This is approximated computing the estimator on trajectories such that $\xi_T$ is larger or respectively smaller than $0$. Again, the blue line shows the theoretical limit density.
In all these cases the CLT is Gaussian and we do not need to have a fine discretization/time grid, since the local time is asymptotically negligible and the quantities which matter in the estimator are $\xi_T$ and the occupation times. This accounts of \eqref{eq:N1}-positive part,
 \eqref{eq:T0}, \eqref{eq:T1p} and~\eqref{eq:T1n}.}
    \end{center}
\end{figure}

\begin{figure}[ht!]
    \begin{center}
	\includegraphics[page=3]{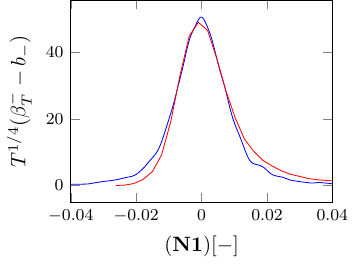}
\caption{
    \textbf{(N1)}$[-]$. SDE parameters: $\sigma_\pm=0.01;$ in case \textbf{N1}: 
$b_-=\num{0.004}; b_+=0$.
Simulation parameters: $T=10^3,N=10^5$.
The CLT in \eqref{eq:N1}-negative part is accurate for large $T$ and time step $T/N$ small, since the quality of the estimation of the local time is key in this case. Remark that in this case (null recurrent), the CLT has speed of convergence $T^{1/4}$ and the limit law is not Gaussian.  This accounts of \eqref{eq:N1}-negative part: we show in red the density of $T^{1/4}(\disbeta_{T,N}^--b_-)$, in blue its limit density.
}
    \end{center}
\end{figure}

\begin{figure}[ht!]
    \begin{center}
	\includegraphics[page=4]{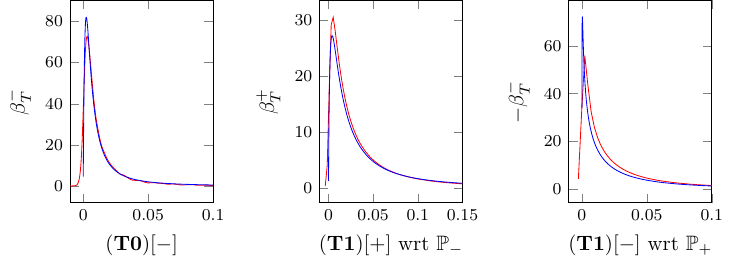}
\caption{
    \textbf{(T0)}$[-]$, \textbf{(T1)}$[+]$ w.r.t to $\PP_-$ and 
    \textbf{(T1)}$[-]$ w.r.t to $\PP_+$. SDE parameters: $\sigma_\pm=0.01$;
in case~\textbf{T0}: $b_-=\num{0.004},\, b_+=\num{0.003}$;
 in case~\textbf{T1}: $b_-=\num{-0.003},\, b_+=\num{0.01}$. 
 In case~\textbf{T0}: simulation parameters: $T=\num{20},\, N=10^4$.
 We display in red the density of $\beta^-_T$ . In case~\textbf{T1}: simulation parameters: 
 $T=\num{20},\, N=10^5$. 
 We display in red the density of~$\disbeta_{T,N}^\pm$ w.r.t $\PP_\mp$ (cf. \eqref{prop:T1}).
 This is approximated computing the estimator on trajectories such that $\xi_T$
 is smaller or respectively larger than $0$. In these cases the estimator is
 not consistent, so what we show is not actually a CLT but the convergence of
 the estimators towards the law \eqref{eq:dens:r}, whose density is plotted in blue (cf. results
 \eqref{eq:finite:T0}, \eqref{eq:finite:T1p}, \eqref{eq:finite:T1n}).  This
 convergence is accurate for large~$T$ but also depends on the time step~$T/N$.
 Moreover, we see that the theoretical distribution of $\beta_T^-$ in
 case~\textbf{T1} is almost singular at the origin, and therefore the exact
 behavior near the origin is hard to catch on simulated trajectories. This can
 be improved using different kernels (instead of the Gaussian one) in the
 estimation of the density. This can be easily done with the function
 \textquote{density} in \texttt{R}. The limit behavior is better approximated when
 $b_+$ and $b_-$ have similar magnitude. We choose here to display the case
 $b_-=\num{-0.003}$; $b_+=\num{0.01}$ to mention this critical behavior. Anyway,
 this feature does not really matter in statistical application, because in
 this case the estimator not only is not consistent, but does not even guess
 the correct sign of the parameter. Indeed, we are here in the very critical
 case of a transient process  generated by diverging drift \textbf{T1}.
}    
 \end{center}
\end{figure}

\begin{figure}[ht!]
    \begin{center}
	\includegraphics[page=5]{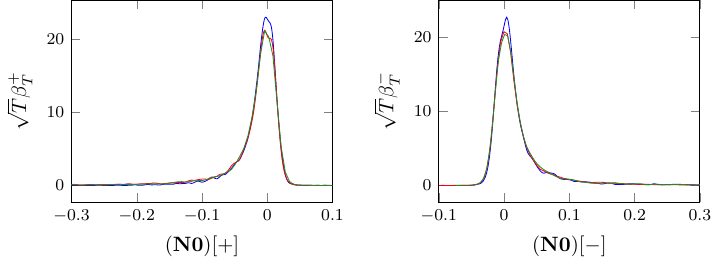}
\caption{
    \textbf{(N0)}. SDE parameters: $\sigma_\pm=0.01$, $b_-=0$, $b_+=0$.
Simulation parameters: $T=10,\,N=10^3$; $T=10^2,\,N=10^4$; $T=10^3,\,N=10^5$.
Differently from before, we do not show the convergence to the scaled limit law  \eqref{explicitdensitynodrift}, but the scaling relation \eqref{eq:scaling:N0}: the blue, red and green lines represent the density of $\sqrt{T}\disbeta_{T,N}^\pm$, for the three different final times. 
We show it on both positive and negative semiaxes.
Because of the estimation of the local time, this also depends on the choice of $N$.}
    \end{center}
\end{figure}



\printbibliography

\end{document}